\documentclass[amstex,12pt]{article}
\usepackage{amssymb,amsmath,amsthm}
\newtheorem{theorem}{Theorem}[section]

\newtheorem{lemma}{Lemma}[section]
\newtheorem{definition}{Definition}[section]
\newtheorem{remark}{Remark}[section]
\newtheorem{example}{Example}[section]

\newcommand{\beq}{\begin{equation}}
\newcommand{\eeq}{\end{equation}}
\newcommand{\beqn}{\begin{eqnarray}}
\newcommand{\eeqn}{\end{eqnarray}}

\allowdisplaybreaks \baselineskip 20pt
\begin{document}

\title{Almost automorphic funtions on time scales  and   almost automorphic solutions to  shunting inhibitory cellular neural networks   on time scales\thanks{This work is supported by
the National Natural Sciences Foundation of People's Republic of
China under Grant 11361072.} }
\author{ Yongkun Li\thanks{%
The corresponding author.}, Bing Li and Xiaofang Meng\\
%EndAName
Department of Mathematics, Yunnan University\\
Kunming, Yunnan 650091\\
People's Republic of China}

\date{}
\maketitle \allowdisplaybreaks
\begin{abstract}
In this paper, we first propose a new concept of almost periodic time scales,   a new definition of almost
automorphic functions on almost periodic time scales, and  study some   their basic properties.
Then we prove a result ensuring the existence of an almost automorphic solution for both the linear nonhomogeneous dynamic equation on time scales and its associated homogeneous equation, assuming that the  associated homogeneous equation  admits an exponential dichotomy.
Finally, as an application of our results, we establish  the existence and global exponential stability of almost automorphic solutions to a class of shunting inhibitory cellular neural networks  with time-varying delays on time scales. Our results about the shunting inhibitory cellular neural network  with time-varying delays on time scales are new even for the both cases of differential equations(the time scale $\mathbb{T}=\mathbb{R})$  and difference equations(the time scale $\mathbb{T}=\mathbb{Z})$.
\end{abstract}

\textbf{Key words:} Time scales; Almost automorphic functions; Dynamic equations; Exponential dichotomy.\\
\textbf{MSC2010:} 34N05; 34K14; 43A60; 92B20.

\allowdisplaybreaks
\section{Introduction}

\setcounter{section}{1}
\setcounter{equation}{0}
\indent

The theory of time scales was initiated by  Hilger \cite{5} in his Ph.D. thesis in 1988, which can   unify the continuous and discrete cases.
The theory of dynamic equations on time scale contains, links and extends the classical theory
of differential and difference equations. In the recent years, there has been an increasing interest in studying
 the existence of periodic solutions and almost periodic solutions of various dynamic equations on time scales; we
refer the reader to the papers [2-23].

The concept of almost automorphy  was introduced in the literature by S. Bochner in
1955 in the context of differential geometry \cite{b1} (see also Bochner \cite{b2,b3}). Since
then, this concept has been extended in various directions.

\begin{remark}Although every almost periodic function is almost automorphic,
the converse, however, is not true.
\end{remark}

In order to study  almost periodic, pseudo almost periodic  and almost automorphic dynamic equations on time scales, the following  concept of
almost periodic time scales was proposed in \cite{9}.
\begin{definition}\label{da1}\cite{9} A time scale $\mathbb{T}$ is called an almost periodic time scale if
\begin{eqnarray*}
\Pi=\big\{\tau\in\mathbb{R}: t\pm\tau\in\mathbb{T}, \forall t\in{\mathbb{T}}\big\}\neq\{0\}.
\end{eqnarray*}
\end{definition}
 Based on Definition \ref{da1}, almost periodic
functions \cite{9}, pseudo almost periodic functions \cite{b31}, almost automorphic functions \cite{b32,b33} and weighted piecewise pseudo almost automorphic functions \cite{b34} on time scales were defined  successfully.
For example, the authors of \cite{b33}  proposed the following concept of almost automorphic functions on time scales.
\begin{definition}\label{da2}\cite{b33}
Let $\mathbb{T}$ be an almost periodic time scale and $\mathbb{X}$ be a Banach space. A bounded continuous function $f:\mathbb{T}\rightarrow\mathbb{X}$ is said to be almost automorphic, if for every sequence $(s_{n}^{'})\subset \Pi$, there is a subsequence $(s_{n})\subset(s_{n}^{'})$ such that
    \[\lim\limits_{n\rightarrow\infty}f(t+s_{n})=\bar{f}(t)\]
    is well defined for each $t\in {\mathbb{T}}$,  and
    \[\lim\limits_{n\rightarrow\infty}\bar{f}(t-s_{n})=f(t)\]
    for each $t\in {\mathbb{T}}$.

    A continuous function $f: \mathbb{T}\times \mathbb{X}\rightarrow \mathbb{X}$ is said to be almost automorphic if $f(t,x)$ is almost automorphic in $t\in \mathbb{T}$ uniformly in $x\in B$, where $B$ is any bounded subset of $\mathbb{X}$.
\end{definition}
Since the concept of almost periodic time scales in sense of Definition \ref{da1}
is very restrictive in the sense that it is a kind of periodic time scale (see \cite{kf}). This excludes many interesting time scales.
Therefore, it is a challenging and important problem in theories and applications to find a
new concept of almost periodic time scales.

Motivated by the above discussions, our main aim of this paper is  to propose a new concept of almost periodic time scales and a new definition of almost
automorphic functions on almost periodic time scales, and  to study  the existence of an almost automorphic solution for both the linear nonhomogeneous dynamic equation on time scales and its associated homogeneous equation.
As an application of our results,   the existence and global exponential stability of almost automorphic solutions to a class of shunting inhibitory cellular neural networks  with time-varying delays on time scales is established.

The organization of this paper is as follows: In Section 2, we introduce some notations and definitions and state some preliminary results which are needed
in later sections. In Section 3, we first introduce a new concept of almost time scales and a new definition of almost automorphic functions on time scales, then we discuss some of their properties. Finally, we propose two open problems.
  In Section 4,  we prove a result ensuring the existence of an almost automorphic solution for both the linear nonhomogeneous dynamic equation on time scales and its associated homogeneous equation, assuming that the  associated homogeneous equation  admits an exponential dichotomy. In Section 5, as an application of our results, we establish  the existence and global exponential stability of almost automorphic solutions to a class of shunting inhibitory cellular neural networks  with time-varying delays on time scales. Finally, we draw a conclusion in Section 6.

\section{Preliminaries}

\setcounter{section}{2}
\setcounter{equation}{0}
\indent

In this section, we shall first recall some fundamental definitions, lemmas which are used in what follows.

A time scale $\mathbb{T}$
is an arbitrary nonempty closed subset of the real set $\mathbb{R}$ with the topology and ordering inherited from $\mathbb{R}$. The
forward jump operator
$\sigma:\mathbb{T}\rightarrow\mathbb{T}$ is defined by $\sigma(t)=\inf\big\{s\in \mathbb{T},s>t\big\}$ for all $t\in \mathbb{T}$,
while the backward jump operator $\rho:\mathbb{T}\rightarrow\mathbb{T}$ is defined by
$\rho(t)=\sup\big\{s\in \mathbb{T},s<t\big\}$ for all $t\in\mathbb{T}$. Finally, the graininess function
$\mu: \mathbb{T}\rightarrow [0,\infty)$ is defined by $\mu(t)=\sigma(t)-t$.
A point $t\in\mathbb{T}$ is called left-dense if $t>\inf\mathbb{T}$
and $\rho(t)=t$, left-scattered if $\rho(t)<t$, right-dense if
$t<\sup\mathbb{T}$ and $\sigma(t)=t$, and right-scattered if
$\sigma(t)>t$. If $\mathbb{T}$ has a left-scattered maximum $m$,
then $\mathbb{T}^k=\mathbb{T}\setminus\{m\}$; otherwise
$\mathbb{T}^k=\mathbb{T}$. If $\mathbb{T}$ has a right-scattered
minimum $m$, then $\mathbb{T}_k=\mathbb{T}\setminus\{m\}$; otherwise
$\mathbb{T}_k=\mathbb{T}$.
A function $f:\mathbb{T}\rightarrow\mathbb{R}$ is rd-continuous provided
it is continuous at right-dense points in $\mathbb{T}$ and its left-side
limits exist at left-dense points in $\mathbb{T}$. The set of all rd-continuous
functions $f:\mathbb{T}\rightarrow\mathbb{R}$ will be denoted by $C_{rd}=C_{rd}(\mathbb{T})=C_{rd}(\mathbb{T},\mathbb{R})$.
A function $r:\mathbb{T}\rightarrow\mathbb{R}$ is called regressive
if
$
1+\mu(t)r(t)\neq 0
$
for all $t\in \mathbb{T}^k$. The set of all regressive and
right-dense continuous functions $r:\mathbb{T}\rightarrow\mathbb{R}$ will
be denoted by $\mathcal{R}=\mathcal{R}(\mathbb{T})=\mathcal{R}(\mathbb{T},\mathbb{R})$. We
define the set
$\mathcal{R}^+=\mathcal{R}^+(\mathbb{T},\mathbb{R})=\{r\in \mathcal{R}:1+\mu(t)r(t)>0,\,\,\forall
t\in\mathbb{T}\}$.
Let $A$ be an $m\times n$-matrix-valued function on $\mathbb{T}$. We say that $A$ is rd-continuous on $\mathbb{T}$ if each entry of $A$ is rd-continuous on $\mathbb{T}$. We denote the class of all rd-continuous $m\times n$ matrix-valued functions on $\mathbb{T}$ by  $C_{rd}=C_{rd}(\mathbb{T})=C_{rd}(\mathbb{T},\mathbb{R}^{m\times n})$.
An $n\times n$-matrix-valued function $A$ on a time scale $\mathbb{T}$ is called regressive (with respect to $\mathbb{T}$) provided
$(1+\mu(t)A(t))$ is invertible for all $t\in \mathbb{T}^k$. The set of all regressive and rd-continuous functions is denoted by $\mathcal{R}=\mathcal{R}(\mathbb{T})=\mathcal{R}(\mathbb{T},\mathbb{R}^{n\times n})$.
For more knowledge of time scales, one can refer to  \cite{6,7}.

\begin{definition}$(\cite{s20})$
Let $x\in\mathbb{R}^{n}$ and $A(t)$ be an $n\times n$
rd-continuous matrix on $\mathbb{T}$, the linear system
\begin{eqnarray}\label{e21}
x^{\Delta}(t)=A(t)x(t),\,\, t\in\mathbb{T}
\end{eqnarray}
is said to admit an exponential dichotomy on $\mathbb{T}$ if there
exist positive constants $k, \alpha$, a projection $P$ and a
fundamental solution matrix $X(t)$ of \eqref{e21}, satisfying
\begin{eqnarray*}
|X(t)PX^{-1}(\sigma(s))|\leq ke_{\ominus
\alpha}(t,\sigma(s)),\,\,
s, t \in\mathbb{T}, t \geq \sigma(s),\\
|X(t)(I-P)X^{-1}(\sigma(s))|\leq ke_{\ominus
\alpha}(\sigma(s),t),\,\, s, t \in\mathbb{T}, t \leq \sigma(s),
\end{eqnarray*}
where $|\cdot|$ is a matrix norm on $\mathbb{T}$, that is $A=(a_{ij})_{n\times n}$, then we can take
$|A|=(\sum\limits_{i=1}^{n}\sum\limits_{j=1}^{n}|a_{ij}|^{2})^{\frac{1}{2}})$.
\end{definition}

\begin{definition}\cite{li}\label{d}
Let $\mathbb{T}_1$ and $\mathbb{T}_2$ be two time scales, we define
\[
\mathrm{dist}(\mathbb{T}_1,\mathbb{T}_2)=\max\{\sup\limits_{t\in \mathbb{T}_1}\big\{\mathrm{dist}(t,\mathbb{T}_2)\big\},\sup\limits_{t\in \mathbb{T}_2}\big\{\mathrm{dist}(t,\mathbb{T}_1)\big\}\}
\]
where, $\mathrm{dist}(t,\mathbb{T}_2)=\inf\limits_{s\in \mathbb{T}_2}\{|t-s|\},\mathrm{dist}(t,\mathbb{T}_1)=\inf\limits_{s\in \mathbb{T}_1}\{|t-s|\}$. Let $\tau\in \mathbb{R}$ and $\mathbb{T}$ be a time scale, we define
\[
\mathrm{dist}(\mathbb{T},\mathbb{T}_\tau)=\max\{\sup\limits_{t\in \mathbb{T}}\big\{\mathrm{dist}(t,\mathbb{T}_\tau)\big\},\sup\limits_{t\in \mathbb{T}_\tau}\big\{\mathrm{dist}(t,\mathbb{T})\big\}\},
\]
where $\mathbb{T}_\tau:=\mathbb{T}\cap\{\mathbb{T}-\tau \}=\mathbb{T}\cap\{t-\tau: \forall t\in \mathbb{T}\}$, $\mathrm{dist}(t,\mathbb{T}_\tau)=\inf\limits_{s\in \mathbb{T}_\tau}\{|t-s|\},\mathrm{dist}(t,\mathbb{T})=\inf\limits_{s\in \mathbb{T}}\{|t-s|\}$.
\end{definition}
\begin{definition}\cite{li}\label{def33}
A time scale $\mathbb{T}$ is called an almost periodic time scale if for every $\varepsilon>0$, there exists a constant $l(\varepsilon)>0$ such that each interval of length $l(\varepsilon)$ contains a $\tau(\varepsilon)$ such that $\mathbb{T}_\tau\neq \emptyset$ and
$
\mathrm{dist}(\mathbb{T},\mathbb{T}_\tau)<\varepsilon,
$
that is, for any $\varepsilon > 0$, the   set
$
\Pi(\mathbb{T},\varepsilon)=\{\tau\in \mathbb{R}, \mathrm{dist}(\mathbb{T},\mathbb{T}_\tau)<\varepsilon\}
$
is relatively dense.  $\tau$ is called the $\varepsilon$-translation number of $\mathbb{T}$.
\end{definition}
Obviously, if $\mathbb{T}$ is an almost periodic time scale, then
$\inf\mathbb{T}=-\infty$ and $\sup\mathbb{T}=+\infty,$ if $\mathbb{T}$ is a periodic time scale (see \cite{li}), then $\mathrm{dist}(\mathbb{T},\mathbb{T}_\tau)=0$, that is $\mathbb{T}=\mathbb{T}_\tau$.
\begin{remark}
One can easily see that if a time scale is an almost time scale under Definition \ref{da1}, then it is also an almost time scale under Definition \ref{def33}.
\end{remark}
\begin{lemma}\cite{li}\label{aa1} Let $\mathbb{T}$ be an almost periodic time scale under Definition \ref{def33}, then
\begin{itemize}
  \item [$(i)$]if $\tau\in \Pi(\mathbb{T},\varepsilon)$, then $t+\tau\in \mathbb{T}$ for all $t\in \mathbb{T}_\tau$;
  \item [$(ii)$] if $\varepsilon_1<\varepsilon_2$, then $\Pi(\mathbb{T},\varepsilon_1)\subset \Pi(\mathbb{T},\varepsilon_2)$;
   \item [$(iii)$]if $\tau\in \Pi(\mathbb{T},\varepsilon)$, then $-\tau\in\Pi(\mathbb{T},\varepsilon)$ and $\mathrm{dist}(\mathbb{T}_\tau,\mathbb{T})=\mathrm{dist}(\mathbb{T}_{-\tau},\mathbb{T})$;
    \item [$(iv)$]if  $\tau_1,\tau_2\in \Pi(\mathbb{T},\varepsilon)$, then $\tau_1+ \tau_2\in\Pi(\mathbb{T},2\varepsilon)$.
\end{itemize}
\end{lemma}

\section{Almost time scales and almost automorphic functions}
\setcounter{equation}{0}
\indent

In this section, we first introduce a new concept of almost time scales and a new definition of almost automorphic functions on time scales, then we discuss some of their properties.

\begin{definition}\label{n1}
A time scale $\mathbb{T}$ is called an almost periodic time scale if
\begin{itemize}
\item[$(i)$] $\Pi:=\{\tau\in \mathbb{R}:\mathbb{T}_{\tau}\neq\emptyset\}\neq \{0\}$ and $\widetilde{\mathbb{T}}\neq\emptyset$,
\item[$(ii)$] if $\tau_{1}, \tau_{2}\in \Pi$, then $\tau_{1}\pm\tau_{2}\in\Pi$,
\end{itemize}
where $\mathbb{T}_{\tau}=\mathbb{T}\cap\{\mathbb{T}-\tau\}=\mathbb{T}\cap\{t-\tau: \forall t\in \mathbb{T}\}$ and $\widetilde{\mathbb{T}}=\bigcap\limits_{\tau\in\Pi}\mathbb{T}_{\tau}$.
\end{definition}
It is obvious that if $\tau\in \Pi$, then $\pm\tau\in \Pi$ and $t\pm\tau\in \mathbb{T}$ for all $t\in \widetilde{\mathbb{T}}$.
\begin{remark}Noticing the fact that $\Pi=\{\tau\in \mathbb{R}: \mathbb{T}_\tau\neq \emptyset\}\supset \Pi(\mathbb{T},\varepsilon)$,
from Lemma \ref{aa1}, one can easily see that if a time scale is an almost time scale under Definition \ref{def33}, then it is also an almost time scale under Definition \ref{n1}.
\end{remark}

\begin{definition}\label{n21}
Let $\mathbb{T}$ be an almost periodic time scale and $\mathbb{X}$ be a Banach space. A bounded rd-continuous function $f:\mathbb{T}\rightarrow\mathbb{X}$ is said to be almost automorphic, if for every sequence $(s_{n}^{'})\subset \Pi$, there is a subsequence $(s_{n})\subset(s_{n}^{'})$ such that
    \[\lim\limits_{n\rightarrow\infty}f(t+s_{n})=\bar{f}(t)\]
    is well defined for each $t\in\widetilde{\mathbb{T}}$,  and
    \[\lim\limits_{n\rightarrow\infty}\bar{f}(t-s_{n})=f(t)\]
    for each $t\in\widetilde{\mathbb{T}}$.
\end{definition}
Denote by $AA(\mathbb{T},\mathbb{X})$ the set of all almost automorphic functions on time scale $\mathbb{T}$.
\begin{remark}
From Definition \ref{n21}, one can easily see that if a function $f:\mathbb{T} \rightarrow\mathbb{X}$ is an  almost automorphic function under Definition \ref{da2}, then it is also an almost automorphic function under Definition \ref{n21}.
\end{remark}

\begin{lemma} Let $\mathbb{T}$ be an almost periodic time scale and suppose $f,f_{1},f_{2}\in AA(\mathbb{T},\mathbb{X})$. Then we have the following
\begin{itemize}
\item[$(i)$] $f_{1}+f_{2}\in AA(\mathbb{T},\mathbb{X})$;
\item[$(ii)$] $\alpha f\in AA(\mathbb{T},\mathbb{X})$ for any constant $\alpha\in\mathbb{R}$;
\item[$(iii)$] $f_{c}(t)\equiv f(c+t)\in AA(\mathbb{T},\mathbb{X})$  for each fixed $c\in\widetilde{\mathbb{T}}$.
\end{itemize}
\end{lemma}
\begin{proof} $(i)$ Let $f_{1},f_{2}\in AA(\mathbb{T},\mathbb{X})$. Then, for every sequence $(s_{n}^{'})\subset \Pi$, there is a subsequence $(s_{n})\subset(s_{n}^{'})$ such that
\begin{eqnarray*}
\lim\limits_{n\rightarrow\infty}f_{1}(t+s_{n})=\bar{f}_{1}(t)\,\,\,\,\mathrm{and}\,\,\,\,\lim\limits_{n\rightarrow\infty}f_{2}(t+s_{n})
=\bar{f}_{2}(t)
\end{eqnarray*}
are well defined for each $t\in\widetilde{\mathbb{T}}$, and
\begin{eqnarray*}
\lim\limits_{n\rightarrow\infty}\bar{f}_{1}(t-s_{n})=f_{1}(t)\,\,\,\,\mathrm{and}\,\,\,\,\lim\limits_{n\rightarrow\infty}\bar{f}_{2}(t-s_{n})
=f_{2}(t)
\end{eqnarray*}
for each $t\in\widetilde{\mathbb{T}}$. Therefore, we obtain
\begin{eqnarray*}
\lim\limits_{n\rightarrow\infty}(f_{1}+f_{2})(t+s_{n})=\lim\limits_{n\rightarrow\infty}\big(f_{1}(t+s_{n})+f_{2}(t+s_{n})\big)
=\bar{f}_{1}(t)+\bar{f}_{2}(t)
\end{eqnarray*}
is well defined for each $t\in\widetilde{\mathbb{T}}$, and
\begin{eqnarray*}
\lim\limits_{n\rightarrow\infty}(\bar{f}_{1}+\bar{f}_{2})(t-s_{n})
=\lim\limits_{n\rightarrow\infty}\big(\bar{f}_{1}(t-s_{n})+\bar{f}_{2}(t-s_{n})\big)
=f_{1}(t)+f_{2}(t)
\end{eqnarray*}
for each $t\in\widetilde{\mathbb{T}}$.

$(ii)$ Since $f\in AA(\mathbb{T},\mathbb{X})$, then for every sequence $(s_{n}^{'})\subset \Pi$, there is a subsequence $(s_{n})\subset(s_{n}^{'})$ such that
\begin{eqnarray*}
\lim\limits_{n\rightarrow\infty}(\alpha f)(t+s_{n})=\lim\limits_{n\rightarrow\infty}\alpha f(t+s_{n})=\alpha\bar{f}(t)=(\alpha\bar{f})(t)
\end{eqnarray*}
is well defined for each $t\in\widetilde{\mathbb{T}}$, and
\begin{eqnarray*}
\lim\limits_{n\rightarrow\infty}(\alpha \bar{f})(t-s_{n})=\lim\limits_{n\rightarrow\infty}\alpha\bar{f}(t-s_{n})=\alpha f(t)=(\alpha f)(t)
\end{eqnarray*}
for each $t\in\widetilde{\mathbb{T}}$.

$(iii)$ It follows from $f\in AA(\mathbb{T},\mathbb{X})$, $c\in\widetilde{\mathbb{T}}$ that for every sequence $(s_{n}^{'})\subset \Pi$, there is a subsequence $(s_{n})\subset(s_{n}^{'})$ such that
\begin{eqnarray*}
\lim\limits_{n\rightarrow\infty}f_{c}(t+s_{n})=\lim\limits_{n\rightarrow\infty}f((c+t)+s_{n})=\bar{f}(c+t)=\bar{f}_{c}(t)
\end{eqnarray*}
is well defined for each $t\in\widetilde{\mathbb{T}}$, and
\begin{eqnarray*}
\lim\limits_{n\rightarrow\infty}\bar{f}_{c}(t-s_{n})=\lim\limits_{n\rightarrow\infty}\bar{f}((c+t)-s_{n})=f(c+t)=f_{c}(t)
\end{eqnarray*}
for each $t\in\widetilde{\mathbb{T}}$.
The proof  is completed.
\end{proof}

\begin{lemma} Let $\mathbb{T}$ be an almost periodic time scale and  functions $f,\phi:\mathbb{T}\rightarrow\mathbb{X}$ be almost automorphic, then the function $\phi f:\mathbb{T}\rightarrow\mathbb{X}$ defined by $(\phi f)(t)=\phi(t)f(t)$ is also almost automorphic.
\end{lemma}
\begin{proof} Both functions $\phi$ and $f$ are bounded since they are almost automorphic. So we let $K_{1}=\sup\limits_{t\in\mathbb{T}}\|\phi(t)\|$.

Let the sequence $(s_{n}^{'})\subset \Pi$, then there exists a subsequence $(s_{n}^{''})\subset(s_{n}^{'})$ such that $\lim\limits_{n\rightarrow\infty}\phi(t+s_{n}^{''})=\bar{\phi}(t)$ is well defined for each $t\in\widetilde{\mathbb{T}}$ and $\lim\limits_{n\rightarrow\infty}\bar{\phi}(t-s_{n}^{''})=\phi(t)$ for each $t\in\widetilde{\mathbb{T}}$. Since $f$ is almost automorphic, there exiats a subsequence $(s_{n})\subset(s_{n}^{''})$ such that $\lim\limits_{n\rightarrow\infty}f(t+s_{n})=\bar{f}(t)$ is well defined for each $t\in\widetilde{\mathbb{T}}$ and $\lim\limits_{n\rightarrow\infty}\bar{f}(t-s_{n})=f(t)$ for each $t\in\widetilde{\mathbb{T}}$.

Now, we have
\begin{eqnarray*}
&&\|\phi(t+s_{n})f(t+s_{n})-\bar{\phi}(t)\bar{f}(t)\|\\
&\leq&\|\phi(t+s_{n})f(t+s_{n})-\phi(t+s_{n})\bar{f}(t)\|+\|\phi(t+s_{n})\bar{f}(t)-\bar{\phi}(t)\bar{f}(t)\|\\
&\leq&K_{1}\|f(t+s_{n})-\bar{f}(t)\|+K_{2}\|\phi(t+s_{n})-\bar{\phi}(t)\|\\
&\leq&(K_{1}+K_{2})\varepsilon
\end{eqnarray*}
for $n$ sufficiently large, where $K_{2}=\sup\limits_{t\in\mathbb{T}}\|\bar{f}(t)\|<\infty$.
Thus, we obtain
\[\lim\limits_{n\rightarrow\infty}\phi(t+s_{n})f(t+s_{n})=\bar{\phi}(t)\bar{f}(t)\]
for each $t\in\widetilde{\mathbb{T}}$.

It is also easy to check that
\[\lim\limits_{n\rightarrow\infty}\bar{\phi}(t-s_{n})\bar{f}(t-s_{n})=\phi(t)f(t)\]
for each $t\in\widetilde{\mathbb{T}}$. The proof is now complete.
\end{proof}

\begin{lemma} Let $\mathbb{T}$ be an almost periodic time scale and $(f_{n})$ be a sequence of almost automorphic functions such that $\lim\limits_{n\rightarrow\infty}f_{n}(t)=f(t)$ converges uniformly for each $t\in\widetilde{\mathbb{T}}$. Then, $f$ is an almost automorphic function.
\end{lemma}
\begin{proof} Let the sequence $(s_{n}^{'})\subset \Pi$. As in the standard case of the almost automorphic functions the approach follows across the diagonal procedure. Since $f_{1}\in AA(\mathbb{T},\mathbb{X})$, then there exists a subsequence $(s_n^{(1)})\subset(s_{n}^{'})$ such that
\begin{eqnarray*}
\lim\limits_{n\rightarrow\infty}f_{1}(t+s_n^{(1)})=\bar{f_{1}}(t)
\end{eqnarray*}
is well defined for each $t\in\widetilde{\mathbb{T}}$, and
\begin{eqnarray*}
\lim\limits_{n\rightarrow\infty}\bar{f_{1}}(t-s_n^{(1)})=f_{1}(t)
\end{eqnarray*}
for each $t\in\widetilde{\mathbb{T}}$. Since $f_{2}\in AA(\mathbb{T},\mathbb{X})$, then there exists a subsequence $(s_n^{(2)})\subset(s_n^{(1)})$ such that
\begin{eqnarray*}
\lim\limits_{n\rightarrow\infty}f_{2}(t+s_n^{(2)})=\bar{f_{2}}(t)
\end{eqnarray*}
is well defined for each $t\in\widetilde{\mathbb{T}}$, and
\begin{eqnarray*}
\lim\limits_{n\rightarrow\infty}\bar{f_{2}}(t-s_n^{(2)})=f_{2}(t)
\end{eqnarray*}
for each $t\in\widetilde{\mathbb{T}}$. Following, by this procedure, we can construct a subsequence $(s_n^{(n)})\subset(s_{n}^{'})$ such that
\begin{eqnarray}\label{ff1}
\lim\limits_{n\rightarrow\infty}f_{i}(t+s_n^{(n)})=\bar{f_{i}}(t)
\end{eqnarray}
for each $t\in\widetilde{\mathbb{T}}$ and for all $i=1,2,\ldots$. Let us consider
\begin{eqnarray}\label{ff2}
\|\bar{f_{i}}(t)-\bar{f_{j}}(t)\|&\leq&\|\bar{f_{i}}(t)-f_{i}(t+s_n^{(n)})\|+\|f_{i}(t+s_n^{(n)})-f_{j}(t+s_n^{(n)})\|\nonumber\\
&&+\|f_{j}(t+s_n^{(n)})-\bar{f_{j}}(t)\|.
\end{eqnarray}
By the uniformly convergence of $(f_{n})$, for any $\varepsilon>0$, we can find $N=N(\varepsilon)\in\mathbb{N}$ sufficiently large such that for all $i,j>N$, we have
\begin{eqnarray}\label{ff3}
\|f_{i}(t+s_n^{(n)})-f_{j}(t+s_n^{(n)})\|<\varepsilon
\end{eqnarray}
for all $t\in\widetilde{\mathbb{T}}$ and all $n=1,2,\ldots$.

Hence, taking $i,j$ sufficiently large in \eqref{ff2} and using \eqref{ff3} and the limit \eqref{ff1}, we can conclude that $(\bar{f_{i}}(t))$ is a Cauchy sequence. Since $\mathbb{X}$ is a Banach space, then $(\bar{f_{i}}(t))$ is a sequence which converges pointwisely on $\mathbb{X}$. Let $\bar{f}(t)$ be the limit of $(\bar{f_{i}}(t))$, then for each $i=1,2,\ldots$, we have
\begin{eqnarray}\label{ff4}
\|f(t+s_n^{(n)})-\bar{f}(t)\|&\leq&\|f(t+s_n^{(n)})-f_{i}(t+s_n^{(n)})\|+\|f_{i}(t+s_n^{(n)})-\bar{f_{i}}(t)\|\nonumber\\
&&+\|\bar{f_{i}}(t)-\bar{f}(t)\|.
\end{eqnarray}
Then, for $i$ sufficiently large, by \eqref{ff4} and using the almost automorphicity of $f_{i}$ and the convergence of the functions $f_{i}$ and $\bar{f_{i}}$, we obtain
\[\lim\limits_{n\rightarrow\infty}f(t+s_n^{(n)})=\bar{f}(t)\]
for each $t\in\widetilde{\mathbb{T}}$. Similarly, we can get
\[\lim\limits_{n\rightarrow\infty}\bar{f}(t-s_n^{(n)})=f(t)\]
for each $t\in\widetilde{\mathbb{T}}$. This completes the proof.
\end{proof}

\begin{lemma} Let $\mathbb{T}$ be an almost periodic time scale and $\mathbb{X},\mathbb{Y}$ be Banach spaces. If $f:\mathbb{T}\rightarrow\mathbb{X}$ is an almost automorphic function  and $\varphi:\mathbb{X}\rightarrow\mathbb{Y}$ is a continuous function, then the composite function $\varphi\circ f:\mathbb{T}\rightarrow\mathbb{Y}$ is an almost automorphic function.
\end{lemma}
\begin{proof}
Since $f\in AA(\mathbb{T},\mathbb{X})$, then for every sequence $(s_{n}^{'})\subset \Pi$, there exists a subsequence $(s_{n})\subset(s_{n}^{'})$ such that $\lim\limits_{n\rightarrow\infty}f(t+s_{n})=\bar{f}(t)$ is well defined for each $t\in\widetilde{\mathbb{T}}$ and
$\lim\limits_{n\rightarrow\infty}\bar{f}(t-s_{n})=f(t)$ for each $t\in\widetilde{\mathbb{T}}$.

By the continuity of function $\varphi$, it follows that
\[\lim\limits_{n\rightarrow\infty}\varphi(f(t+s_{n}))=\varphi\big(\lim\limits_{n\rightarrow\infty}f(t+s_{n})\big)
=(\varphi\circ\bar{f})(t).\]
On the other hand, we can get
\[\lim\limits_{n\rightarrow\infty}\varphi(\bar{f}(t-s_{n}))=\varphi\big(\lim\limits_{n\rightarrow\infty}\bar{f}(t-s_{n})\big)
=(\varphi\circ f)(t)\]
for each $t\in\widetilde{\mathbb{T}}$. Thus, $\varphi\circ f\in AA(\mathbb{T},\mathbb{Y})$. The proof is  complete.
\end{proof}

\begin{definition}\label{da3}
Let $\mathbb{T}$ be an almost periodic time scale and $\mathbb{X}$ be a Banach space. A bounded rd-continuous function $f:\mathbb{T}\times\mathbb{X}\rightarrow\mathbb{X}$ is said to be almost automorphic, if for any sequence $(s_{n}^{'})\subset \Pi$, there is a subsequence $(s_{n})\subset(s_{n}^{'})$ such that
    \[\lim\limits_{n\rightarrow\infty}f(t+s_{n},x)=\bar{f}(t,x)\]
    is well defined for each $t\in\widetilde{\mathbb{T}}$, $t+s_{n}\in\mathbb{T}$, $x\in\mathbb{X}$, and
    \[\lim\limits_{n\rightarrow\infty}\bar{f}(t-s_{n},x)=f(t,x)\]
    for each $t\in\widetilde{\mathbb{T}}$, $x\in\mathbb{X}$.
\end{definition}
Denote by $AA(\mathbb{T}\times\mathbb{X},\mathbb{X})$ the set of all such functions.
\begin{remark}
From Definition \ref{da3}, one can easily see that if a function $f:\mathbb{T}\times\mathbb{X}\rightarrow\mathbb{X}$ is an  almost automorphic function under Definition \ref{da2}, then it is also an almost automorphic function under Definition \ref{da3}.
\end{remark}

\begin{lemma}
Let $f\in AA(\mathbb{T}\times\mathbb{X},\mathbb{X})$ and $f$ satisfy the Lipschitz condition in $x\in\mathbb{X}$ uniformly in $t\in\widetilde{\mathbb{T}}$. If $\varphi\in AA(\mathbb{T},\mathbb{X})$, then $f(t,\varphi(t))$ is almost automorphic.
\end{lemma}
\begin{proof}
For any sequence $(s_{n}^{'})\subset \Pi$, there is a subsequence $(s_{n})\subset(s_{n}^{'})$ such that
 \[\lim\limits_{n\rightarrow\infty}f(t+s_{n},x)=\bar{f}(t,x)\]
 is well defined for each $t\in\widetilde{\mathbb{T}}$, $x\in\mathbb{X}$, and
    \[\lim\limits_{n\rightarrow\infty}\bar{f}(t-s_{n},x)=f(t,x)\]
    for each $t\in\widetilde{\mathbb{T}}$, $x\in\mathbb{X}$. Since $\varphi\in AA(\mathbb{T},\mathbb{X})$, there exists a subsequence $(\tau_{n})\subset(s_{n})$ such that
\begin{eqnarray*}
\lim\limits_{n\rightarrow\infty}\varphi(t+\tau_{n})=\bar{\varphi}(t)
\end{eqnarray*}
is well defined for each $t\in\widetilde{\mathbb{T}}$, and
\begin{eqnarray*}
\lim\limits_{n\rightarrow\infty}\bar{\varphi}(t-\tau_{n})=\varphi(t)
\end{eqnarray*}
for each $t\in\widetilde{\mathbb{T}}$. Since $f$ satisfies the Lipschitz condition in $x\in\mathbb{X}$ uniformly in $t\in\widetilde{\mathbb{T}}$, there exists a positive constant $L$ such that
\begin{eqnarray*}
&&\|f(t+\tau_{n},\varphi(t+\tau_{n}))-\bar{f}(t,\bar{\varphi}(t))\|\\
&\leq&\|f(t+\tau_{n},\varphi(t+\tau_{n}))-f(t+\tau_{n},\bar{\varphi}(t))\|+\|f(t+\tau_{n},\bar{\varphi}(t))-\bar{f}(t,\bar{\varphi}(t))\|\\
&\leq&L\|\varphi(t+\tau_{n})-\bar{\varphi}(t)\|+\|f(t+\tau_{n},\bar{\varphi}(t))-\bar{f}(t,\bar{\varphi}(t))\|
\rightarrow0,\,\,\,n\rightarrow\infty
\end{eqnarray*}
and
\begin{eqnarray*}
&&\|\bar{f}(t-\tau_{n},\bar{\varphi}(t-\tau_{n}))-f(t,\varphi(t))\|\\
&\leq&\|\bar{f}(t-\tau_{n},\bar{\varphi}(t-\tau_{n}))-\bar{f}(t-\tau_{n},\varphi(t))\|+\|\bar{f}(t-\tau_{n},\varphi(t))-f(t,\varphi(t))\|\\
&\leq&L\|\bar{\varphi}(t-\tau_{n})-\varphi(t)\|+\|\bar{f}(t-\tau_{n},\varphi(t))-f(t,\varphi(t))\|\rightarrow0,\,\,\,n\rightarrow\infty.
\end{eqnarray*}
Hence, for any sequence $(s_{n}^{'})\subset \Pi$, there is a subsequence $(\tau_{n})\subset(s_{n}^{'})$ such that
    \[\lim\limits_{n\rightarrow\infty}f(t+\tau_{n},\varphi(t+\tau_{n}))=\bar{f}(t,\bar{\varphi}(t))\]
    is well defined for each $t\in\widetilde{\mathbb{T}}$, and
    \[\lim\limits_{n\rightarrow\infty}\bar{f}(t-\tau_{n},\bar{\varphi}(t-\tau_{n}))=f(t,\varphi(t))\]
    for each $t\in\widetilde{\mathbb{T}}$. That is, $f(t,\varphi(t))$ is almost automorphic. This completes the proof.
\end{proof}

\begin{definition}\label{n2}
Let $\mathbb{T}$ be an almost periodic time scale. The graininess function $\mu:\mathbb{T}\rightarrow\mathbb{R_+}$ is said to be almost automorphic, if for every sequence $(s_{n}^{'})\subset \Pi$, there is a subsequence $(s_{n})\subset(s_{n}^{'})$ such that
    \[\lim\limits_{n\rightarrow\infty}\mu(t+s_{n})=\bar{\mu}(t)\]
    is well defined for each $t\in\widetilde{\mathbb{T}}$,  and
    \[\lim\limits_{n\rightarrow\infty}\bar{\mu}(t-s_{n})=\mu(t)\]
    for each $t\in\widetilde{\mathbb{T}}$.
\end{definition}

In the following,
in order to make the graininess function $\mu$ have a better property, we will use Definition \ref{def33} as the definition of  almost periodic time scales.

From Corollary 15, Theorem 19 in \cite{li} and Definition \ref{n2},  we can obtain
\begin{lemma}\label{au}
Let $\mathbb{T}$ be an almost periodic time scale under Definition \ref{def33}. Then the graininess function $\mu$ is almost automorphic.
\end{lemma}
\noindent
\textbf{Open problem 1:} {\it{Let $\mathbb{T}$ be an almost periodic time scale under Definition \ref{n1}. Whether or not the graininess function $\mu$ is almost automorphic?}}

\noindent
\textbf{Open problem 2:} {\it{Let  the graininess function $\mu$ of $\mathbb{T}$ be almost automorphic. Whether or not $\mathbb{T}$ is an almost periodic time scale under Definition \ref{n1}?}}

\section{Automorphic solutions to linear dynamic equations}
\setcounter{equation}{0}

\indent

Consider the linear nonhomogeneous dynamic equation on time scales
\begin{eqnarray}\label{e22}
x^{\Delta}(t)=A(t)x(t)+f(t),\,\, t \in \mathbb{T},
\end{eqnarray}
where $A:\mathbb{T}\rightarrow \mathbb{R}^{n\times n}, f:\mathbb{T}\rightarrow \mathbb{R}^n$ and its associated homogeneous equation
 \begin{eqnarray}\label{e221}
x^{\Delta}(t)=A(t)x(t),\,\, t \in \mathbb{T}.
\end{eqnarray}
Throughout this section, we restrict our discussions on almost periodic time scales under Definition \ref{def33} and we assume that $A(t)$ is almost automorphic on $\mathbb{T}$, which means that  each entry of the matrix $A(t)$ is almost automorphic.

\begin{lemma} Let $\mathbb{T}$ be an almost periodic time scale, $A(t)\in \mathcal{R}(\mathbb{T},\mathbb{R}^{n\times n})$ be almost automorphic and nonsingular on $\mathbb{T}$ and the sets $\{A^{-1}(t)\}_{t\in\mathbb{T}}$ and $\{(I+\mu(t)A(t))^{-1}\}_{t\in\mathbb{T}}$ be bounded on $\mathbb{T}$. Then, $A^{-1}(t)$ and $(I+\mu(t)A(t))^{-1}$ are almost automorphic on $\mathbb{T}$. Moreover, suppose that $f\in AA(\mathbb{T},\mathbb{R}^{n})$ and \eqref{e221} admits an exponential dichotomy, then  \eqref{e22}
has a solution $x(t)\in AA(\mathbb{T},\mathbb{R}^{n})$ and $x(t)$ is expressed as follows
\begin{eqnarray}\label{e231}
x(t)=\int_{-\infty}^{t}X(t)PX^{-1}(\sigma(s))f(s)\Delta
s-\int_t^{+\infty}X(t)(I-P)X^{-1}(\sigma(s))f(s)\Delta s,
\end{eqnarray}
where $X(t)$ is the fundamental solution matrix of \eqref{e221}.
\end{lemma}
\begin{proof}
 We divide the proof into several steps.

Step 1. $A^{-1}(t)$ is almost automorphic on $\mathbb{T}$.

Let the sequence $(s_{n}^{'})\subset \Pi$. Since $A(t)$ is almost automorphic on time scales, there exists a subsequence $(s_{n})\subset(s_{n}^{'})$ such that
\[\lim\limits_{n\rightarrow\infty}A(t+s_{n})=\bar{A}(t)\]
is well defined for each $t\in\widetilde{\mathbb{T}}$, and
\[\lim\limits_{n\rightarrow\infty}\bar{A}(t-s_{n})=A(t)\]
for each $t\in\widetilde{\mathbb{T}}$.

Given $t\in\widetilde{\mathbb{T}}$ and define $A_{n}=A(t+s_{n})$, $n\in\mathbb{N}$. By hypothesis, the set $\{A_{n}^{-1}\}_{n\in\mathbb{N}}$ is bounded, that is, there exists a positive constant $M$ such that $|A_{n}^{-1}|<M$. From the identity $A_{n}^{-1}-A_{m}^{-1}=A_{n}^{-1}(A_{m}-A_{n})A_{m}^{-1}$, it implies that $\{A_{n}\}$ is a cauchy sequence, in addition, it follows that $\{A_{n}^{-1}\}$ is cauchy sequence. Hence, for each $t\in\widetilde{\mathbb{T}}$ fixed, there exists a matrix $S$ such that
\[A_{n}^{-1}(t)\rightarrow S(t).\]
Then we have
\[\lim\limits_{n\rightarrow\infty}A_{n}A_{n}^{-1}=\bar{A}\bar{A}^{-1}=I,\]
where $I$ denotes the identity matrix, we obtain that $\bar{A}(t)$ is invertible and $\bar{A}^{-1}(t)=S(t)$ for each $t\in\widetilde{\mathbb{T}}$. Since the map $A\rightarrow A^{-1}$ is continuous on the set of nonsingular matrices, we can obtain
\[\lim\limits_{n\rightarrow\infty}A^{-1}(t+s_{n})=\bar{A}^{-1}(t)\]
is well defined for each $t\in\widetilde{\mathbb{T}}$. Similarly, we can get
\[\lim\limits_{n\rightarrow\infty}\bar{A}^{-1}(t-s_{n})=A^{-1}(t)\]
for each $t\in\widetilde{\mathbb{T}}$.

Step 2. $(I+\mu(t)A(t))^{-1}$ is almost automorphic on $\mathbb{T}$.

Since $A(t)$ and $\mu(t)$ are almost automorphic functions, then for every sequence $(s_{n}^{'})\subset \Pi$, there exists a subsequence $(s_{n})\subset(s_{n}^{'})$ such that
\begin{eqnarray*}
\lim\limits_{n\rightarrow\infty}A(t+s_{n})=\bar{A}(t)\,\,\,\,\ \mathrm{and}\,\,\,\,\ \lim\limits_{n\rightarrow\infty}\mu(t+s_{n})=\bar{\mu}(t)
\end{eqnarray*}
are well defined for each $t\in\widetilde{\mathbb{T}}$, and
\begin{eqnarray*}
\lim\limits_{n\rightarrow\infty}\bar{A}(t-s_{n})=A(t)\,\,\,\,\ \mathrm{and}\,\,\,\,\ \lim\limits_{n\rightarrow\infty}\bar{\mu}(t-s_{n})=\mu(t)
\end{eqnarray*}
for each $t\in\widetilde{\mathbb{T}}$. Therefore, we have
\[\lim\limits_{n\rightarrow\infty}\big(I+A(t+s_{n})\mu(t+s_{n})\big)=I+\bar{A}(t)\bar{\mu}(t)\]
for each $t\in\widetilde{\mathbb{T}}$, and
\[\lim\limits_{n\rightarrow\infty}\big(I+\bar{A}(t-s_{n})\bar{\mu}(t-s_{n})\big)=I+A(t)\mu(t)\]
for each $t\in\widetilde{\mathbb{T}}$. Hence, $(I+A(t)\mu(t))$ is almost automorphic on $\mathbb{T}$. In addition, since $A(t)$ is a regressive matrix,  $(I+A(t)\mu(t))$ is nonsingular on $\mathbb{T}$. By hypothesis, the set $\{(I+A(t)\mu(t))^{-1}\}_{t\in\mathbb{T}}$ is bounded. The proof is similar to the proof of Step 1, we obtain that $(I+A(t)\mu(t))^{-1}$ is almost automorphic on $\mathbb{T}$.

Step 3. The system \eqref{e22} has an automorphic solution.

Since the linear system \eqref{e221} admits an exponential dichotomy,   system \eqref{e22} has a bounded solution $x(t)$ given by
\[
x(t)=\int_{-\infty}^{t}X(t)PX^{-1}(\sigma(s))f(s)\Delta
s-\int_t^{+\infty}X(t)(I-P)X^{-1}(\sigma(s))f(s)\Delta s.
\]
Since $A(t)$, $\mu(t)$ and $f(t)$ are almost automorphic functions, then for every sequence $(s_{n}^{'})\subset \Pi$, there exists a subsequence $(s_{n})\subset(s_{n}^{'})$ such that
\begin{eqnarray*}
\lim\limits_{n\rightarrow\infty}A(t+s_{n})=\bar{A}(t),\,\,\,\,\ \lim\limits_{n\rightarrow\infty}\mu(t+s_{n})=\bar{\mu}(t)\,\,\,\,\  and\,\,\,\,\ \lim\limits_{n\rightarrow\infty}f(t+s_{n})=\bar{f}(t)
\end{eqnarray*}
are well defined for each $t\in\widetilde{\mathbb{T}}$, and
\begin{eqnarray*}
\lim\limits_{n\rightarrow\infty}\bar{A}(t-s_{n})=A(t),\,\,\,\,\ \lim\limits_{n\rightarrow\infty}\bar{\mu}(t-s_{n})=\mu(t)\,\,\,\,\ and\,\,\,\,\ \lim\limits_{n\rightarrow\infty}\bar{f}(t-s_{n})=f(t)
\end{eqnarray*}
for each $t\in\widetilde{\mathbb{T}}$.

We let
\[B(t)=\int_{-\infty}^{t}X(t)PX^{-1}(\sigma(s))f(s)\Delta s\]
and
\[\bar{B}(t)=\int_{-\infty}^{t}M(t,s)\bar{f}(s)\Delta s,\]
where $M(t,s)=\lim\limits_{n\rightarrow\infty}X(t+s_{n})PX^{-1}(\sigma(s+s_{n}))$.

Then, we obtain
\begin{eqnarray}\label{e26}
&&\|B(t+s_{n})-\bar{B}(t)\|\nonumber\\
&=&\Big\|\int_{-\infty}^{t+s_{n}}X(t+s_{n})PX^{-1}(\sigma(s))f(s)\Delta s-\int_{-\infty}^{t}M(t,s)\bar{f}(s)\Delta s\Big\|\nonumber\\
&=&\Big\|\int_{-\infty}^{t}X(t+s_{n})PX^{-1}(\sigma(s+s_{n}))f(s+s_{n})\Delta s-\int_{-\infty}^{t}M(t,s)\bar{f}(s)\Delta s\Big\|\nonumber\\
&\leq&\Big\|\int_{-\infty}^{t}X(t+s_{n})PX^{-1}(\sigma(s+s_{n}))f(s+s_{n})\Delta s-\int_{-\infty}^{t}X(t+s_{n})PX^{-1}(\sigma(s+s_{n}))\bar{f}(s)\Delta s\Big\|\nonumber\\
&&+\Big\|\int_{-\infty}^{t}X(t+s_{n})PX^{-1}(\sigma(s+s_{n}))\bar{f}(s)\Delta s-\int_{-\infty}^{t}M(t,s)\bar{f}(s)\Delta s\Big\|\nonumber\\
&=&\Big\|\int_{-\infty}^{t}X(t+s_{n})PX^{-1}(\sigma(s+s_{n}))\big(f(s+s_{n})-\bar{f}(s)\big)\Delta s\Big\|\nonumber\\
&&+\Big\|\int_{-\infty}^{t}\big(X(t+s_{n})PX^{-1}(\sigma(s+s_{n}))-M(t,s)\big)\bar{f}(s)\Delta s\Big\|.
\end{eqnarray}
In addition, since the function $f$ is almost automorphic,   we have $\bar{f}$ is a bounded function, limit both sides with \eqref{e26}, then we can get
\begin{equation}\label{es1}
  \lim\limits_{n\rightarrow\infty}B(t+s_{n})=\bar{B}(t)
\end{equation}
for each $t\in\widetilde{\mathbb{T}}$. Analogously, one can prove that
\begin{equation}\label{es2}
   \lim\limits_{n\rightarrow\infty}\bar{B}(t-s_{n})=B(t)
\end{equation}
for each $t\in\widetilde{\mathbb{T}}$.

On the other hand, let
\[C(t)=\int_t^{+\infty}X(t)(I-P)X^{-1}(\sigma(s))f(s)\Delta s\]
and
\[\bar{C}(t)=\int_{-\infty}^{t}N(t,s)\bar{f}(s)\Delta s,\]
where $N(t,s)=\lim\limits_{n\rightarrow\infty}X(t+s_{n})(I-P)X^{-1}(\sigma(s+s_{n}))$. Then,
  similar to the proofs of \eqref{es1} and \eqref{es2},  we can prove that given a sequence $(s_{n}^{'})\subset \Pi$, there exists a subsequence $(s_{n})\subset(s_{n}^{'})$ such that
\[\lim\limits_{n\rightarrow\infty}C(t+s_{n})=\bar{C}(t)\]
for each $t\in\widetilde{\mathbb{T}}$, and
\[\lim\limits_{n\rightarrow\infty}\bar{C}(t-s_{n})=C(t)\]
for each $t\in\widetilde{\mathbb{T}}$.

Finally, we define $\bar{x}(t)=\bar{B}(t)-\bar{C}(t)$, then by the definition of $x$ from \eqref{e231}, one can prove that given a sequence $(s_{n}^{'})\subset \Pi$, there exists a subsequence $(s_{n})\subset(s_{n}^{'})$ such that
\[\lim\limits_{n\rightarrow\infty}x(t+s_{n})=\bar{x}(t)\]
is well defined for each $t\in\widetilde{\mathbb{T}}$, and
\[\lim\limits_{n\rightarrow\infty}\bar{x}(t-s_{n})=x(t)\]
for each $t\in\widetilde{\mathbb{T}}$. Hence, $x(t)$ is an almost automorphic solution to system \eqref{e22}. This completes the proof.
\end{proof}

Similar to the proof of Lemma 2.15 in \cite{li2}, one can easily show that
\begin{lemma}
Let $c_{i}(t)$ be   almost  automorphic   on $\mathbb{T}$, where
$c_{i}(t)>0, -c_{i}(t)\in \mathcal{R}^{+},  t\in
\widetilde{\mathbb{T}}, i=1,2,\ldots,n$ and $\min\limits_{1 \leq i \leq n}\{\inf\limits_{t\in
\widetilde{\mathbb{T}}}c_i(t)\}=\widetilde{m}> 0$, then the linear system
\begin{eqnarray*}\label{e23}
x^{\Delta}(t)=\mathrm{diag}(-c_{1}(t),-c_{2}(t),\dots,-c_{n}(t))x(t)
\end{eqnarray*}
admits an exponential dichotomy on $\mathbb{T}$.
\end{lemma}

\section{An application}

\setcounter{equation}{0}

\indent

In the last forty years, shunting inhibitory cellular neural
networks (SICNNs) have been extensively applied in psychophysics,
speech, perception, robotics, adaptive pattern
recognition, vision, and image processing. Hence, they have
been the object of intensive analysis by numerous authors
in recent years.
Many important results on
the dynamical behaviors of SICNNs have been established and successfully applied to signal processing, pattern recognition,
associative memories, and so on. We refer the reader to [37-44] and the references cited therein. However, to the best of our knowledge, there is no paper published on the existence of almost automorphic solutions to   SICNNs governed by differential or difference equations.
So,
our main purpose of this section is  to study the existence of almost automorphic solutions
to the following SICNN  with time-varying delays on time scales
\begin{eqnarray}\label{s1}
x_{ij}^{\Delta}(t)&=&-a_{ij}(t)x_{ij}(t)-\sum\limits_{C_{kl}\in N_{r}(i,j)}B_{ij}^{kl}(t)f(x_{kl}(t-\tau_{kl}(t)))x_{ij}(t)\nonumber\\
&&-\sum\limits_{C_{kl}\in N_{p}(i,j)}C_{ij}^{kl}(t)\int_{t-\delta_{kl}(t)}^{t}g(x_{kl}(u))\Delta u x_{ij}(t)+L_{ij}(t),
\end{eqnarray}
where $t\in \mathbb{T}$, $\mathbb{T}$ is an almost periodic time scale under the sense of Definition \ref{def33}, $i=1,2,\ldots,m$, $j=1,2,\ldots,n$, $C_{ij}$ denotes the cell at the $(i,j)$ position of the lattice, the $r$-neighborhood $N_{r}(i,j)$ of $C_{ij}$ is given by
\[N_{r}(i,j)=\{C_{kl}:\max(|k-i|,|l-j|)\leq r, 1\leq k\leq m, 1\leq l\leq n\},\]
$N_{p}(i,j)$ is similarly specified, $x_{ij}$ is the activity of the cell $C_{ij}$, $L_{ij}(t)$ is the external input to $C_{ij}$, $a_{ij}(t)>0$ represents the passive decay rate of the cell activity, $B_{ij}^{kl}(t)\geq0$ and $C_{ij}^{kl}(t)\geq0$
are the connections or coupling strengths of postsynaptic activity of the cells in $N_{r}(i,j)$ and $N_{p}(i,j)$ transmitted to cell $C_{ij}$ depending upon variable delays and continuously distributed delays, respectively, and the activity functions $f(\cdot)$ and $g(\cdot)$ are continuous functions representing the output or firing rate of the cell $C_{kl}$, $\tau_{kl}(t)$ and $\delta_{kl}(t)$ are transmission delays at time $t$ and satisfy $t-\tau_{kl}(t)\in \mathbb{T}$ and $t-\delta_{kl}(t)\in\mathbb{T}$ for $t\in\mathbb{T}$, $k=1,2,\ldots,m$, $l=1,2,\ldots,n$.

The initial condition associated with system \eqref{s1} is of the form
\begin{eqnarray*}\label{s2}
x_{ij}(s)=\varphi_{ij}(s),\,\,\,\,s\in[-\theta,0]_{\mathbb{T}},\,\,\,i=1,2,\ldots,m, j=1,2,\ldots,n,
\end{eqnarray*}
where $\theta=\max\big\{\max\limits_{1\leq k\leq m, 1\leq l\leq n}{\overline{\tau_{kl}}},\max\limits_{1\leq k\leq m, 1\leq l\leq n}{\overline{\delta_{kl}}}\big\}$, $\varphi_{ij}(\cdot)$ denotes a real-value bounded rd-continuous function defined on $[-\theta,0]_{\mathbb{T}}$.

Throughout this section, we let $[a,b]_{\mathbb{T}}=\{t|t\in[a,b]\cap\mathbb{T}\}$ and we restrict our discussions on almost periodic time scales under the sense of Definition \ref{def33}. For convenience, for an almost automorphic function  $f:\mathbb{T}\rightarrow\mathbb{R}$, denote
$\underline{f}=\inf\limits_{t\in \mathbb{T}}f(t), \overline{f}=\sup\limits_{t\in\mathbb{T}}f(t).$
We denote by $\mathbb{R}$ the set of real numbers, by $\mathbb{R}^{+}$ the set of positive real numbers.

Set
\[x=\{x_{ij}(t)\}=(x_{11}(t),\ldots,x_{1n}(t),\ldots,x_{i1}(t),\ldots,x_{in}(t),\ldots,x_{m1}(t)\ldots,x_{mn}(t))\in \mathbb{R}^{m\times n}.\]
For $\forall x=\{x_{ij}(t)\} \in \mathbb{R}^{m\times n}$, we define the norm $\|x(t)\|=\max_{i,j}|x_{ij}(t)|$ and $\|x\|=\sup\limits_{t\in\mathbb{T}}\|x(t)\|$.

Let $\mathbb{B}=\{\varphi|\varphi=\{\varphi_{ij}(t)\}
=(\varphi_{11}(t),\ldots,\varphi_{1n}(t),\ldots,\varphi_{i1}(t),\ldots,\varphi_{in}(t),\ldots,\varphi_{m1}(t)\ldots,\varphi_{mn}(t))\}$, where $\varphi$ is an almost automorphic function on $\mathbb{T}$. For $\forall\varphi\in\mathbb{B}$, if we define the norm
$\|\varphi\|_{\mathbb{B}}=\sup\limits_{t\in\mathbb{T}}\|\varphi(t)\|$, then $\mathbb{B}$ is a Banach space.

\begin{definition}
The almost automorphic solution
$x^{\ast}(t)=\{x_{ij}^{\ast}(t)\}$ of system \eqref{s1} with initial value $\varphi^{\ast}=\{\varphi_{ij}^{\ast}(t)\}$ is said to be globally exponentially stable. If there exist  positive constants $\lambda$ with $\ominus\lambda\in \mathcal{R}^+$ and $M>1$
such that every solution $x(t)=\{x_{ij}(t)\}$ of system \eqref{s1} with initial value $\varphi=\{\varphi_{ij}(t)\}$ satisfies
\[
\|x-x^{\ast}\|_{\mathbb{B}}\leq M
e_{\ominus\lambda}(t,t_0)\|\varphi-\varphi^{\ast}\|_{\mathbb{B}},\,\,\,\,\forall
t\in[t_{0},+\infty)_{\mathbb{T}},\,\,\,t_{0}\in\mathbb{T}.
\]
\end{definition}

\begin{theorem}\label{thm31} Suppose that
\begin{itemize}
\item[$(H_{1})$] for $ij\in \Lambda=\{11,12,\ldots,1n,\ldots,m1,m2,\ldots,mn\}$, $-a_{ij}\in\mathcal{R}^{+}$, where $\mathcal {R}^{+}$ denotes the set of positively regressive functions from $\mathbb{T}$ to $\mathbb{R}$, $a_{ij},B_{ij}^{kl},C_{ij}^{kl},\tau_{kl},L_{ij}\in \mathbb{B}$;
\item[$(H_{2})$] functions $f,g\in C(\mathbb{R},\mathbb{R})$ and there exist positive constants $L^{f},L^{g},M^{f},M^{g}$ such that for all $u,v\in \mathbb{R}$,
\[
|f(u)-f(v)|\leq L^{f}|u-v|,\,\,\,f(0)=0,\,\,\,|f(u)|\leq M^{f},
\]
\[
|g(u)-g(v)|\leq L^{g}|u-v|,\,\,\,g(0)=0,\,\,\,|g(u)|\leq M^{g};
\]
\item[$(H_{3})$] there exists a constant $\rho$ such that
\[\max\limits_{i,j}\bigg\{\frac{\sum\limits_{C_{kl}\in N_{r}(i,j)}\overline{B_{ij}^{kl}}M^{f}\rho^{2}
+\sum\limits_{C_{kl}\in N_{p}(i,j)}\overline{C_{ij}^{kl}}M^{g}\overline{\delta_{kl}}\rho^{2}+\overline{L_{ij}}}{\underline{a_{ij}}}
\bigg\}<\rho\] and
\[\max\limits_{i,j}\bigg\{\frac{\sum\limits_{C_{kl}\in N_{r}(i,j)}\overline{B_{ij}^{kl}}(M^{f}+L^{f})\rho
+\sum\limits_{C_{kl}\in N_{p}(i,j)}\overline{C_{ij}^{kl}}(M^{g}+L^{g})\overline{\delta_{kl}}\rho}{\underline{a_{ij}}}\bigg\}<1.\]
\end{itemize}
Then, system \eqref{s1} has a unique almost automorphic solution in $\mathbb{E}=\{\varphi\in \mathbb{B}:\|\varphi\|_{\mathbb{B}}\leq \rho\}$.
\end{theorem}
\begin{proof}
For any given $\varphi\in\mathbb{B}$, we consider the following system
\begin{eqnarray}\label{s3}
x_{ij}^{\Delta}(t)&=&-a_{ij}(t)x_{ij}(t)-\sum\limits_{C_{kl}\in N_{r}(i,j)}B_{ij}^{kl}(t)f(\varphi_{kl}(t-\tau_{kl}(t)))\varphi_{ij}(t)\nonumber\\
&&\sum\limits_{C_{kl}\in N_{p}(i,j)}C_{ij}^{kl}(t)\int_{t-\delta_{kl}(t)}^{t}g(\varphi_{kl}(u))\Delta u \varphi_{ij}(t)+L_{ij}(t),\,\,ij=11,12,\ldots,mn.
\end{eqnarray}
Since $\min\limits_{1\leq i\leq m; 1\leq j\leq n}\{\inf\limits_{t\in\mathbb{T}}a_{ij}(t)\}>0$ and $-a_{ij}\in\mathcal{R}^{+}$, it follows from Lemma 2.12 that the linear system
\begin{eqnarray*}\label{s4}
x_{ij}^{\Delta}(t)=-a_{ij}(t)x_{ij}(t),\,\,\,\,ij=11,12,\ldots,mn
\end{eqnarray*}
admits an exponential dichotomy on $\mathbb{T}$. Thus, by Lemma 2.11, we obtain that system \eqref{s3} has exactly one almost automorphic
solution as follows
\begin{eqnarray*}
x^{\varphi}(t):=\{x_{ij}^{\varphi}(t)\}&=&\bigg\{\int_{-\infty}^{t}e_{-a_{ij}}(t,\sigma(s))
\Big(-\sum\limits_{C_{kl}\in N_{r}(i,j)}B_{ij}^{kl}(s)f(\varphi_{kl}(s-\tau_{kl}(s)))\varphi_{ij}(s)\\
&&-\sum\limits_{C_{kl}\in N_{p}(i,j)}C_{ij}^{kl}(s)\int_{t-\delta_{kl}(s)}^{s}g(\varphi_{kl}(u))\Delta u \varphi_{ij}(s)+L_{ij}(s)\Big)\Delta s\bigg\}.
\end{eqnarray*}
Now, we define the following operator $T: \mathbb{B}\rightarrow\mathbb{B}$ by setting
\begin{eqnarray*}
T(\varphi(t))=x^{\varphi}(t),\,\,\,\,\forall \varphi\in\mathbb{B}.
\end{eqnarray*}
First, we will check that for any $\varphi\in\mathbb{E}$,  $T\varphi\in\mathbb{E}$. For any $\varphi\in\mathbb{E}$, we have
\begin{eqnarray*}
\|T(\varphi)\|_{\mathbb{B}}&=&\sup\limits_{t\in\mathbb{T}}\max\limits_{i,j}\bigg\{\bigg|\int_{-\infty}^{t}e_{-a_{ij}}(t,\sigma(s))
\Big(-\sum\limits_{C_{kl}\in N_{r}(i,j)}B_{ij}^{kl}(s)f(\varphi_{kl}(s-\tau_{kl}(s)))\varphi_{ij}(s)\\
&&-\sum\limits_{C_{kl}\in N_{p}(i,j)}C_{ij}^{kl}(s)\int_{t-\delta_{kl}(s)}^{s}g(\varphi_{kl}(u))\Delta u \varphi_{ij}(s)+L_{ij}(s)\Big)\Delta s\bigg|\bigg\}\\
&\leq&\sup\limits_{t\in\mathbb{T}}\max\limits_{i,j}\bigg\{\bigg|\int_{-\infty}^{t}e_{-\underline{a_{ij}}}(t,\sigma(s))
\Big(\sum\limits_{C_{kl}\in N_{r}(i,j)}\overline{B_{ij}^{kl}}f(\varphi_{kl}(s-\tau_{kl}(s)))\varphi_{ij}(s)\\
&&+\sum\limits_{C_{kl}\in N_{p}(i,j)}\overline{C_{ij}^{kl}}\int_{t-\delta_{kl}(s)}^{s}g(\varphi_{kl}(u))\Delta u \varphi_{ij}(s)\Big)\Delta s\bigg|\bigg\}+\max\limits_{i,j}\frac{\overline{L_{ij}}}{\underline{a_{ij}}}\\
&\leq&\sup\limits_{t\in\mathbb{T}}\max\limits_{i,j}\bigg\{\int_{-\infty}^{t}e_{-\underline{a_{ij}}}(t,\sigma(s))
\Big(\sum\limits_{C_{kl}\in N_{r}(i,j)}\overline{B_{ij}^{kl}}M^{f}|\varphi_{kl}(s-\tau_{kl}(s))||\varphi_{ij}(s)|\\
&&+\sum\limits_{C_{kl}\in N_{p}(i,j)}\overline{C_{ij}^{kl}}M^{g}\int_{t-\delta_{kl}(s)}^{s}|\varphi_{kl}(u)|\Delta u |\varphi_{ij}(s)|\Big)\Delta s\bigg\}+\max\limits_{i,j}\frac{\overline{L_{ij}}}{\underline{a_{ij}}}\\
&\leq&\sup\limits_{t\in\mathbb{T}}\max\limits_{i,j}\bigg\{\int_{-\infty}^{t}e_{-\underline{a_{ij}}}(t,\sigma(s))
\Big(\sum\limits_{C_{kl}\in N_{r}(i,j)}\overline{B_{ij}^{kl}}M^{f}\rho^{2}\\
&&+\sum\limits_{C_{kl}\in N_{p}(i,j)}\overline{C_{ij}^{kl}}M^{g}\overline{\delta_{kl}}\rho^{2}\Big) \Delta s\bigg\}+\max\limits_{i,j}\frac{\overline{L_{ij}}}{\underline{a_{ij}}}\\
&\leq&\max\limits_{i,j}\bigg\{\frac{\sum\limits_{C_{kl}\in N_{r}(i,j)}\overline{B_{ij}^{kl}}M^{f}\rho^{2}
+\sum\limits_{C_{kl}\in N_{p}(i,j)}\overline{C_{ij}^{kl}}M^{g}\overline{\delta_{kl}}\rho^{2}+\overline{L_{ij}}}{\underline{a_{ij}}}
\bigg\}
\leq\rho.
\end{eqnarray*}
Therefore, we have that $\|T(\varphi)\|_{\mathbb{B}}\leq \rho$. This implies that the mapping $T$ is a self-mapping from $\mathbb{E}$ to $\mathbb{E}$. Next, we prove that the mapping $T$ is a contraction mapping on $\mathbb{E}$. In fact, for any $\varphi,\psi\in \mathbb{E}$, we can get
\begin{eqnarray*}
&&\|T(\varphi)-T(\psi)\|_{\mathbb{B}}\\
&=&\sup\limits_{t\in\mathbb{T}}\|T(\varphi)(t)-T(\psi)(t)\|\\
&=&\sup\limits_{t\in\mathbb{T}}\max\limits_{i,j}\bigg\{\bigg|\int_{-\infty}^{t}e_{-a_{ij}}(t,\sigma(s))
\Big(\sum\limits_{C_{kl}\in N_{r}(i,j)}B_{ij}^{kl}(s)\big[f(\varphi_{kl}(s-\tau_{kl}(s)))\varphi_{ij}(s)\\
&&-f(\psi_{kl}(s-\tau_{kl}(s)))\psi_{ij}(s)\big]+\sum\limits_{C_{kl}\in N_{p}(i,j)}C_{ij}^{kl}(s)\big[\int_{t-\delta_{kl}(s)}^{s}g(\varphi_{kl}(u))\Delta u \varphi_{ij}(s)\\
&&-\int_{t-\delta_{kl}(s)}^{s}g(\psi_{kl}(u))\Delta u\psi_{ij}(s)\big]\Big)\Delta s\bigg|\bigg\}\\
&\leq&\sup\limits_{t\in\mathbb{T}}\max\limits_{i,j}\bigg\{\int_{-\infty}^{t}e_{-\underline{a_{ij}}}(t,\sigma(s))
\Big(\sum\limits_{C_{kl}\in N_{r}(i,j)}\overline{B_{ij}^{kl}}\big|f(\varphi_{kl}(s-\tau_{kl}(s)))\big||\varphi_{ij}(s)-\psi_{ij}(s)|\\
&&+\sum\limits_{C_{kl}\in N_{p}(i,j)}\overline{C_{ij}^{kl}}\big|\int_{t-\delta_{kl}(s)}^{s}g(\varphi_{kl}(u))\Delta u\big||\varphi_{ij}(s)-\psi_{ij}(s)|\Big)\Delta s\bigg\}\\
&&+\sup\limits_{t\in\mathbb{T}}\max\limits_{i,j}\bigg\{\int_{-\infty}^{t}e_{-\underline{a_{ij}}}(t,\sigma(s))
\Big(\sum\limits_{C_{kl}\in N_{r}(i,j)}\overline{B_{ij}^{kl}}\big|f(\varphi_{kl}(s-\tau_{kl}(s)))-f(\psi_{kl}(s-\tau_{kl}(s)))\big|\\
&&\times|\psi_{ij}(s)|
+\sum\limits_{C_{kl}\in N_{p}(i,j)}\overline{C_{ij}^{kl}}\big|\int_{t-\delta_{kl}(s)}^{s}\big(g(\varphi_{kl}(u))-g(\psi_{kl}(u))\big)\Delta u\big||\psi_{ij}(s)|
\Big)\bigg\}\\
&\leq&\sup\limits_{t\in\mathbb{T}}\max\limits_{i,j}\bigg\{\int_{-\infty}^{t}e_{-\underline{a_{ij}}}(t,\sigma(s))
\Big(\sum\limits_{C_{kl}\in N_{r}(i,j)}\overline{B_{ij}^{kl}}M^{f}|\varphi_{kl}(s-\tau_{kl}(s))||\varphi_{ij}(s)-\psi_{ij}(s)|\\
&&+\sum\limits_{C_{kl}\in N_{p}(i,j)}\overline{C_{ij}^{kl}}\int_{t-\delta_{kl}(s)}^{s}M^{g}|\varphi_{kl}(u)|\Delta u|\varphi_{ij}(s)-\psi_{ij}(s)|\Big)\Delta s\bigg\}\\
&&+\sup\limits_{t\in\mathbb{T}}\max\limits_{i,j}\bigg\{\int_{-\infty}^{t}e_{-\underline{a_{ij}}}(t,\sigma(s))
\Big(\sum\limits_{C_{kl}\in N_{r}(i,j)}\overline{B_{ij}^{kl}}L^{f}|\varphi_{kl}(s-\tau_{kl}(s))-\psi_{kl}(s-\tau_{kl}(s))|\\
&&\times|\psi_{ij}(s)|
+\sum\limits_{C_{kl}\in N_{p}(i,j)}\overline{C_{ij}^{kl}}\int_{t-\delta_{kl}(s)}^{s}L^{g}|\varphi_{kl}(u)-\psi_{kl}(u)|\Delta u\big||\psi_{ij}(s)|
\Big)\bigg\}\\
&\leq&\sup\limits_{t\in\mathbb{T}}\max\limits_{i,j}\bigg\{\int_{-\infty}^{t}e_{-\underline{a_{ij}}}(t,\sigma(s))
\Big(\sum\limits_{C_{kl}\in N_{r}(i,j)}\overline{B_{ij}^{kl}}M^{f}\rho+\sum\limits_{C_{kl}\in N_{p}(i,j)}\overline{C_{ij}^{kl}}M^{g}\overline{\delta_{kl}}\rho\Big)\Delta s\bigg\}\|\varphi-\psi\|_{\mathbb{B}}\\
&&+\sup\limits_{t\in\mathbb{T}}\max\limits_{i,j}\bigg\{\int_{-\infty}^{t}e_{-\underline{a_{ij}}}(t,\sigma(s))
\Big(\sum\limits_{C_{kl}\in N_{r}(i,j)}\overline{B_{ij}^{kl}}L^{f}\rho+\sum\limits_{C_{kl}\in N_{p}(i,j)}\overline{C_{ij}^{kl}}L^{g}\overline{\delta_{kl}}\rho
\Big)\Delta s\bigg\}\|\varphi-\psi\|_{\mathbb{B}}\\
&\leq&\max\limits_{i,j}\bigg\{\frac{\sum\limits_{C_{kl}\in N_{r}(i,j)}\overline{B_{ij}^{kl}}(M^{f}+L^{f})\rho
+\sum\limits_{C_{kl}\in N_{p}(i,j)}\overline{C_{ij}^{kl}}(M^{g}+L^{g})\overline{\delta_{kl}}\rho}{\underline{a_{ij}}}\bigg\}\|\varphi-\psi\|_{\mathbb{B}}.
\end{eqnarray*}
By $(H_{3})$, we have that $\|T(\varphi)-T(\psi)\|_{\mathbb{B}}<\|\varphi-\psi\|_{\mathbb{B}}$. Hence, $T$ is a contraction mapping from $\mathbb{E}$ to $\mathbb{E}$. Therefore, $T$ has a fixed point in $\mathbb{E}$, that is, \eqref{s1} has a unique almost automorphic solution in $\mathbb{E}$. This completes the proof.
\end{proof}

\begin{theorem}\label{thm32}
Assume that  $(H_{1})$-$(H_{3})$ hold, then system \eqref{s1} has a unique almost automorphic solution which is globally exponentially stable.
\end{theorem}
\begin{proof}
From Theorem \ref{thm31}, we see that system \eqref{s1} has at least one almost automorphic solution
$x^{\ast}(t)=\{x^{\ast}_{ij}(t)\}$ with the initial condition $\varphi^*(t)=\{\varphi_{ij}^*(t)\}$.
Suppose that
$\{x(t)\}=\{(x_{ij}^*(t)\}$
is an arbitrary solution with the initial condition $\varphi(t)=\{\varphi_{ij}(t)\}$. Set $y(t)=x(t)-x^{\ast}(t)$, then it follows from system \eqref{s1} that
\begin{eqnarray}\label{h1}
y_{ij}^{\Delta}(t)&=&-a_{ij}(t)y_{ij}(t)-\sum\limits_{C_{kl}\in N_{r}(i,j)}B_{ij}^{kl}(t)\big(f(x_{kl}(t-\tau_{kl}(t)))x_{ij}(t)\nonumber\\
&&-f(x^{\ast}_{kl}(t-\tau_{kl}(t)))x^{\ast}_{ij}(t)\big)-\sum\limits_{C_{kl}\in N_{p}(i,j)}C_{ij}^{kl}(t)\Big(\int_{t-\delta_{kl}(t)}^{t}g(x_{kl}(u))\Delta u x_{ij}(t)\nonumber\\
&&-\int_{t-\delta_{kl}(t)}^{t}g(x^{\ast}_{kl}(u))\Delta u x^{\ast}_{ij}(t)\Big)
\end{eqnarray}
for $i=1,2,\ldots,m$, $j=1,2,\ldots,n$, the initial condition of \eqref{h1} is
\begin{eqnarray*}
\psi_{ij}(s)=\varphi_{ij}(s)-\varphi^{\ast}_{ij}(s),\,\,\,\,s\in[-\theta,0]_{\mathbb{T}},\,\,\,\,ij=11,12,\ldots,mn.
\end{eqnarray*}
Then, it follows from \eqref{h1} that for $i=1,2,\ldots,m$, $j=1,2,\ldots,n$ and $t\geq t_{0}$, we have
\begin{eqnarray}\label{h0}
y_{ij}(t)&=&y_{ij}(t_{0})e_{-a_{ij}}(t,t_{0})-\int_{t_{0}}^{t}e_{-a_{ij}}(t,\sigma(s))\bigg\{\sum\limits_{C_{kl}\in N_{r}(i,j)}B_{ij}^{kl}(s)\big(f(x_{kl}(s-\tau_{kl}(s)))x_{ij}(s)\nonumber\\
&&-f(x^{\ast}_{kl}(s-\tau_{kl}(s)))x^{\ast}_{ij}(s)\big)+\sum\limits_{C_{kl}\in N_{p}(i,j)}C_{ij}^{kl}(s)\Big(\int_{t-\delta_{kl}(s)}^{s}g(x_{kl}(u))\Delta u x_{ij}(s)\nonumber\\
&&-\int_{s-\delta_{kl}(s)}^{s}g(x^{\ast}_{kl}(u))\Delta u x^{\ast}_{ij}(s)\Big)\bigg\}\Delta s.
\end{eqnarray}

Let $S_{ij}$ be defined as follows:
\begin{eqnarray*}
S_{ij}(\omega)&=&\underline{a_{ij}}-\omega-\exp(\omega\sup\limits_{s\in\mathbb{T}}\mu(s))W_{ij}(\omega),\,\, i=1,2,\ldots,m,j=1,2,\ldots,n,
\end{eqnarray*}
where
\begin{eqnarray*}
W_{ij}(\omega)&=&\Big(\sum\limits_{C_{kl}\in N_{r}(i,j)}\overline{B_{ij}^{kl}}M^{f}\rho
+\sum\limits_{C_{kl}\in N_{p}(i,j)}\overline{C_{ij}^{kl}}(M^{g}+L^{g})\overline{\delta_{kl}}\rho\exp\big\{\omega\overline{\delta_{kl}}\big\}\\
&&+\sum\limits_{C_{kl}\in N_{r}(i,j)}\overline{B_{ij}^{kl}}L^{f}\rho \exp\big\{\omega\overline{\tau_{kl}}\big\}\Big),\,\,\,i=1,2,\ldots,m, j=1,2,\ldots,n.
\end{eqnarray*}
By $(H_{3})$, we get
{\setlength\arraycolsep{2pt}\begin{eqnarray*}
S_{ij}(0)=\underline{a_{ij}}-W_{ij}(0)> 0,\,\, i=1,2,\ldots,m, j=1,2,\ldots,n.
\end{eqnarray*}}
Since $S_{ij}$ is continuous on $[0,+\infty)$ and $S_{ij}(\omega)\rightarrow -\infty$, as
$\omega\rightarrow +\infty$, so there exist $\xi_{ij} > 0$ such that
$S_{ij}(\xi_{ij})=0$ and $S_{ij}(\omega)> 0$
for $\omega\in(0,\xi_{ij})$, $i=1,2,\ldots,m$, $j=1,2,\ldots,n$.
Take
$a=\min\limits_{1\leq i\leq m, 1\leq j\leq n}\big\{\xi_{ij}\big\}$,
we have $S_{ij}(a)\geq 0, i=1,2,\ldots,m, j=1,2,\ldots,n.$ So, we can
choose a positive constant $0< \lambda <
\min\big\{a,\min\limits_{1\leq i \leq m, 1\leq j\leq n}\{\underline{a_{ij}}\}\big\}$ such
that
\[
S_{ij}(\lambda)>0, \quad i=1,2,\ldots,m, j=1,2,\ldots,n,
\]
which implies that
\begin{eqnarray}\label{h2}
\frac{\exp(\lambda\sup\limits_{s\in\mathbb{T}}\mu(s))W_{ij}(\lambda)}{\underline{a_{ij}}-\lambda}< 1,\,\,i=1,2,\ldots,m, j=1,2,\ldots,n.
\end{eqnarray}
Let
\begin{eqnarray*}
M=\max\limits_{1\leq i \leq m, 1\leq j\leq n}\bigg\{\frac{\underline{a_{ij}}}{{W_{ij}(0)}}\bigg\},
\end{eqnarray*}
by $(H_{3})$ and \eqref{h2}, we have $M>1$.

Moreover, we have $e_{\ominus\lambda}(t,t_{0})>1$, where $t\leq t_{0}$. Hence, it is obvious that
\begin{eqnarray*}
\|y(t)\|_{\mathbb{B}}\leq Me_{\ominus\lambda}(t,t_{0})\|\psi\|_{\mathbb{B}}, \qquad \forall
t\in[t_{0}-\theta,t_{0}]_{\mathbb{T}},
\end{eqnarray*}
where $\lambda\in\mathcal{R}^{+}$. In the following , we will show that
\begin{eqnarray}\label{h4}
\|y(t)\|_{\mathbb{B}}\leq Me_{\ominus\lambda}(t,t_{0})\|\psi\|_{\mathbb{B}}, \quad
\forall t\in(t_{0},+\infty)_{\mathbb{T}}.
\end{eqnarray}
To prove \eqref{h4}, we first show that for any $p>1$, the following
inequality holds:
\begin{eqnarray}\label{h5}
\|y(t)\|_{\mathbb{B}}< p Me_{\ominus\lambda}(t,t_{0})\|\psi\|_{\mathbb{B}}, \quad
\forall t\in(t_{0},+\infty)_{\mathbb{T}},
\end{eqnarray}
which implies that, for $i=1,2,\ldots,m$, $j=1,2,\ldots,n$, we have
\begin{eqnarray}\label{h6}
|y_{ij}(t)|< p Me_{\ominus\lambda}(t,t_{0})\|\psi\|_{\mathbb{B}}, \quad
\forall t\in(t_{0},+\infty)_{\mathbb{T}}.
\end{eqnarray}
If \eqref{h6} is not true, then there exists  $t_{1}\in(t_{0},+\infty)_{\mathbb{T}}$ such that
\begin{eqnarray*}
&&|y_{ij}(t_{1})|\geq p Me_{\ominus\lambda}(t_{1},t_{0})\|\psi\|_{\mathbb{B}},\\
&&|y_{ij}(t)|<p Me_{\ominus\lambda}(t,t_{0})\|\psi\|_{\mathbb{B}}, \quad
\forall t\in(t_{0},t_{1})_{\mathbb{T}}.
\end{eqnarray*}
Therefore, there must exist a constant $c\geq 1$ such that
\begin{eqnarray*}
&&|y_{ij}(t_{1})|=cp M e_{\ominus\lambda}(t_1,t_{0})\|\psi\|_{\mathbb{B}},\\
&&|y_{ij}(t)|<cp Me_{\ominus\lambda}(t,t_{0})\|\psi\|_{\mathbb{B}}, \quad
\forall t\in(t_{0},t_{1})_{\mathbb{T}}.
\end{eqnarray*}
Note that, in view of \eqref{h0}, we have that
\begin{eqnarray*}
|y_{ij}(t_{1})|
&=&\bigg|y_{ij}(t_{0})e_{-a_{ij}}(t_{1},t_{0})-\int_{t_{0}}^{t_{1}}e_{-a_{ij}}(t_{1},\sigma(s))\bigg\{\sum\limits_{C_{kl}\in N_{r}(i,j)}B_{ij}^{kl}(s)\big(f(x_{kl}(s-\tau_{kl}(s)))x_{ij}(s)\\
&&-f(x^{\ast}_{kl}(s-\tau_{kl}(s)))x^{\ast}_{ij}(s)\big)+\sum\limits_{C_{kl}\in N_{p}(i,j)}C_{ij}^{kl}(s)\Big(\int_{s-\delta_{kl}(s)}^{s}g(x_{kl}(u))\Delta u x_{ij}(s)\\
&&-\int_{s-\delta_{kl}(s)}^{s}g(x^{\ast}_{kl}(u))\Delta u x^{\ast}_{ij}(s)\Big)\bigg\}\Delta s\bigg|\\
&\leq&e_{-a_{ij}}(t_{1},t_{0})\|\psi\|_{\mathbb{B}}+\int_{t_{0}}^{t_{1}}e_{-a_{ij}}(t_{1},\sigma(s))\bigg\{\sum\limits_{C_{kl}\in N_{r}(i,j)}\overline{B_{ij}^{kl}}\big|f(x_{kl}(s-\tau_{kl}(s)))x_{ij}(s)\\
&&-f(x^{\ast}_{kl}(s-\tau_{kl}(s)))x^{\ast}_{ij}(s)\big|+\sum\limits_{C_{kl}\in N_{p}(i,j)}\overline{C_{ij}^{kl}}\Big|\int_{s-\delta_{kl}(s)}^{s}g(x_{kl}(u))\Delta u x_{ij}(s)\\
&&-\int_{s-\delta_{kl}(s)}^{s}g(x^{\ast}_{kl}(u))\Delta u x^{\ast}_{ij}(s)\Big|\bigg\}\Delta s\\
&\leq&e_{-a_{ij}}(t_{1},t_{0})\|\psi\|_{\mathbb{B}}+\int_{t_{0}}^{t_{1}}e_{-a_{ij}}(t_{1},\sigma(s))\bigg\{\sum\limits_{C_{kl}\in N_{r}(i,j)}\overline{B_{ij}^{kl}}M^{f}|x_{kl}(s-\tau_{kl}(s))||y_{ij}(s)|\\
&&+\sum\limits_{C_{kl}\in N_{p}(i,j)}\overline{C_{ij}^{kl}}\int_{s-\delta_{kl}(s)}^{s}M^{g}|x_{kl}(u)|\Delta u|y_{ij}(s)|\\
&&+\sum\limits_{C_{kl}\in N_{r}(i,j)}\overline{B_{ij}^{kl}}L^{f}|y_{kl}(s-\tau_{kl}(s))||x^{\ast}_{ij}(s)|\\
&&+\sum\limits_{C_{kl}\in N_{p}(i,j)}\overline{C_{ij}^{kl}}\int_{s-\delta_{kl}(s)}^{s}L^{g}|y_{kl}(u)|\Delta u\big||x^{\ast}_{ij}(s)|
\Big)\bigg\}\Delta s\\
&\leq&e_{-a_{ij}}(t_{1},t_{0})\|\psi\|_{\mathbb{B}}+cpM e_{\ominus \lambda}(t_{1},t_{0})\|\psi\|_{\mathbb{B}}\bigg|\int_{t_{0}}^{t_{1}}e_{-a_{ij}\oplus\lambda}(t_{1},\sigma(s))\\
&&\times\Big\{\sum\limits_{C_{kl}\in N_{r}(i,j)}\overline{B_{ij}^{kl}}M^{f}\rho e_{\lambda}(\sigma(s),s)
+\sum\limits_{C_{kl}\in N_{p}(i,j)}\overline{C_{ij}^{kl}}M^{g}\overline{\delta_{kl}}\rho e_{\lambda}(\sigma(s),s-\delta_{kl}(s))\\
&&+\sum\limits_{C_{kl}\in N_{r}(i,j)}\overline{B_{ij}^{kl}}L^{f}\rho e_{\lambda}(\sigma(s),s-\tau_{kl}(s))\\
&&+\sum\limits_{C_{kl}\in N_{p}(i,j)}\overline{C_{ij}^{kl}}L^{g}\overline{\delta_{kl}}\rho e_{\lambda}(\sigma(s),s-\delta_{kl}(s))\big|
\Big\}\bigg|\Delta s\\
&\leq&e_{-a_{ij}}(t_{1},t_{0})\|\psi\|_{\mathbb{B}}+cpM e_{\ominus \lambda}(t_{1},t_{0})\|\psi\|_{\mathbb{B}}\bigg|\int_{t_{0}}^{t_{1}}e_{-a_{ij}\oplus\lambda}(t_{1},\sigma(s))\\
&&\times\Big\{\sum\limits_{C_{kl}\in N_{r}(i,j)}\overline{B_{ij}^{kl}}M^{f}\rho\exp\big\{\lambda\sup\limits_{s\in\mathbb{T}}\mu(s)\big\}\\
&&+\sum\limits_{C_{kl}\in N_{p}(i,j)}\overline{C_{ij}^{kl}}M^{g}\overline{\delta_{kl}}\rho
\exp\big\{\lambda(\overline{\delta_{kl}}+\sup\limits_{s\in\mathbb{T}}\mu(s))\big\}\\
&&+\sum\limits_{C_{kl}\in N_{r}(i,j)}\overline{B_{ij}^{kl}}L^{f}\rho \exp\big\{\lambda(\overline{\tau_{kl}}+\sup\limits_{s\in\mathbb{T}}\mu(s))\big\}\\
&&+\sum\limits_{C_{kl}\in N_{p}(i,j)}\overline{C_{ij}^{kl}}L^{g}\overline{\delta_{kl}}\rho
\exp\big\{\lambda(\overline{\delta_{kl}}+\sup\limits_{s\in\mathbb{T}}\mu(s))\big\}
\Big\}\bigg|\Delta s\\
&=&cpM e_{\ominus \lambda}(t_{1},t_{0})\|\psi\|_{\mathbb{B}}\bigg\{\frac{1}{cpM}e_{-a_{ij}\oplus\lambda}(t_{1},t_{0})
+\exp\big\{\lambda\sup\limits_{s\in\mathbb{T}}\mu(s)\big\}\Big[\sum\limits_{C_{kl}\in N_{r}(i,j)}\overline{B_{ij}^{kl}}M^{f}\rho\\
&&+\sum\limits_{C_{kl}\in N_{p}(i,j)}\overline{C_{ij}^{kl}}(M^{g}+L^{g})\overline{\delta_{kl}}\rho\exp\big\{\lambda\overline{\delta_{kl}}\big\}
+\sum\limits_{C_{kl}\in N_{r}(i,j)}\overline{B_{ij}^{kl}}L^{f}\rho \exp\big\{\lambda\overline{\tau_{kl}}\big\}\Big]\\
&&\times\int_{t_{0}}^{t_{1}}e_{-a_{ij}\oplus\lambda}(t_{1},\sigma(s))\Delta s\\
&<&cpM e_{\ominus \lambda}(t_{1},t_{0})\|\psi\|_{\mathbb{B}}\bigg\{\frac{1}{M}e_{-(\underline{a_{ij}}-\lambda)}(t_{1},t_{0})
+\exp\big\{\lambda\sup\limits_{s\in\mathbb{T}}\mu(s)\big\}\Big[\sum\limits_{C_{kl}\in N_{r}(i,j)}\overline{B_{ij}^{kl}}M^{f}\rho\\
&&+\sum\limits_{C_{kl}\in N_{p}(i,j)}\overline{C_{ij}^{kl}}(M^{g}+L^{g})\overline{\delta_{kl}}\rho\exp\big\{\lambda\overline{\delta_{kl}}\big\}
+\sum\limits_{C_{kl}\in N_{r}(i,j)}\overline{B_{ij}^{kl}}L^{f}\rho \exp\big\{\lambda\overline{\tau_{kl}}\big\}\Big]\\
&&\times\frac{1}{-(\underline{a_{ij}}-\lambda)}\int_{t_{0}}^{t_{1}}(-(\underline{a_{ij}}-\lambda))
e_{-(\underline{a_{ij}}-\lambda)}(t_{1},\sigma(s))\Delta s\\
&=&cpM e_{\ominus \lambda}(t_{1},t_{0})\|\psi\|_{\mathbb{B}}\bigg\{\frac{1}{M}
-\frac{\exp\big\{\lambda\sup\limits_{s\in\mathbb{T}}\mu(s)\big\}}{\underline{a_{ij}}-\lambda}\Big(\sum\limits_{C_{kl}\in N_{r}(i,j)}\overline{B_{ij}^{kl}}M^{f}\rho\\
&&+\sum\limits_{C_{kl}\in N_{p}(i,j)}\overline{C_{ij}^{kl}}(M^{g}+L^{g})\overline{\delta_{kl}}\rho\exp\big\{\lambda\overline{\delta_{kl}}\big\}
+\sum\limits_{C_{kl}\in N_{r}(i,j)}\overline{B_{ij}^{kl}}L^{f}\rho \exp\big\{\lambda\overline{\tau_{kl}}\big\}\Big)\\
&&\times e_{-(\underline{a_{ij}}-\lambda)}(t_{1},t_{0})
+\frac{\exp\big\{\lambda\sup\limits_{s\in\mathbb{T}}\mu(s)\big\}}{\underline{a_{ij}}-\lambda}\Big(\sum\limits_{C_{kl}\in N_{r}(i,j)}\overline{B_{ij}^{kl}}M^{f}\rho\\
&&+\sum\limits_{C_{kl}\in N_{p}(i,j)}\overline{C_{ij}^{kl}}(M^{g}+L^{g})\overline{\delta_{kl}}\rho\exp\big\{\lambda\overline{\delta_{kl}}\big\}
+\sum\limits_{C_{kl}\in N_{r}(i,j)}\overline{B_{ij}^{kl}}L^{f}\rho \exp\big\{\lambda\overline{\tau_{kl}}\big\}\Big)\bigg\}\\
&<&cpM e_{\ominus \lambda}(t_{1},t_{0})\|\psi\|_{\mathbb{B}},
\end{eqnarray*}
which is a contradiction. Therefore, \eqref{h5} holds. Letting $p\rightarrow
1$, then \eqref{h4} holds. Hence, we have that
\begin{eqnarray*}
\|y(t)\|_{\mathbb{B}}< Me_{\ominus\lambda}(t,t_{0})\|\psi\|_{\mathbb{B}}, \quad
\forall t\in(t_{0},+\infty)_{\mathbb{T}},
\end{eqnarray*}
which means that  the almost automorphic solution of system \eqref{s1} is globally exponentially stable. The proof is complete.
\end{proof}

\begin{remark}
Theorems \ref{thm31} and \ref{thm32} are new even for the both cases of differential equations $(\mathbb{T}=\mathbb{R})$  and difference equations $(\mathbb{T}=\mathbb{Z})$.
\end{remark}

\section{ An example}
 \setcounter{equation}{0}

 \indent

In this section, we will give  examples to illustrate the feasibility
and effectiveness of our results obtained in Sections 5.

Consider the following SICNNs with time-varying delays:
\begin{eqnarray}\label{g1}
x_{ij}^{\Delta}(t)&=&-a_{ij}(t)x_{ij}(t)-\sum\limits_{C_{kl}\in N_{r}(i,j)}B_{ij}^{kl}(t)f(x_{kl}(t-\tau_{kl}(t)))x_{ij}(t)\nonumber\\
&&-\sum\limits_{C_{kl}\in N_{p}(i,j)}C_{ij}^{kl}(t)\int_{t-\delta_{kl}(t)}^{t}g(x_{kl}(u))\Delta u x_{ij}(t)+L_{ij}(t),
\end{eqnarray}
where $i=1,2,3$, $j=1,2,3$, $r=p=1$, $f(x)=0.05|\cos x|$, $g(x)=\frac{|x|}{20}$, $t\in\mathbb{T}$.
\begin{example}
If $\mathbb{T}=\mathbb{R}$, then $\mu(t)\equiv0$. Take
\[
\left(
\begin{array}{cccc}
a_{11}(t) & a_{12}(t) & a_{13}(t)\\
a_{21}(t) & a_{22}(t) & a_{23}(t)\\
a_{31}(t) & a_{32}(t) & a_{33}(t)\\
\end{array}\right)=\left(
\begin{array}{cccc}
3+|\sin t| & 4+|\cos \sqrt{2}t| & 2+|\sin t|\\
2+|\cos(\frac{1}{3} t)| & 5+|\sin 2t| & 1+|\cos 2t|\\
4+|\cos2t| & 3+|\sin 2t| & 2+|\cos t|\\
\end{array}\right),\]
\begin{eqnarray*}
\left(
\begin{array}{cccc}
B_{11}(t) & B_{12}(t) & B_{13}(t)\\
B_{21}(t) & B_{22}(t) & B_{23}(t)\\
B_{31}(t) & B_{32}(t) & B_{33}(t)\\
\end{array}\right)&=&\left(
\begin{array}{cccc}
C_{11}(t) & C_{12}(t) & C_{13}(t)\\
C_{21}(t) & C_{22}(t) & C_{23}(t)\\
C_{31}(t) & C_{32}(t) & C_{33}(t)\\
\end{array}\right)\\&=&
\left(
\begin{array}{cccc}
0.1|\cos t| & 0.2|\cos 2t| & 0.1|\sin t|\\
0.1|\sin 2t| & 0.1|\cos t| & 0.2|\cos 2t|\\
0.1|\cos t| & 0.2|\cos 2t| & 0.1|\cos t|\\
\end{array}\right),
\end{eqnarray*}
\[
\left(
\begin{array}{cccc}
L_{11}(t) & L_{12}(t) & L_{13}(t)\\
L_{21}(t) & L_{22}(t) & L_{23}(t)\\
L_{31}(t) & L_{32}(t) & L_{33}(t)\\
\end{array}\right)=\left(
\begin{array}{cccc}
0.1\sin t & \cos t & 0.2\sin t\\
0.2\sin t & 0.4\sin t & 0.1\sin t\\
0.1\sin t & 0.3\sin t & 0.1\cos t\\
\end{array}\right),\]
\[
\left(
\begin{array}{cccc}
\delta_{11}(t) & \delta_{12}(t) & \delta_{13}(t)\\
\delta_{21}(t) & \delta_{22}(t) & \delta_{23}(t)\\
\delta_{31}(t) & \delta_{32}(t) & \delta_{33}(t)\\
\end{array}\right)=\left(
\begin{array}{cccc}
0.5+0.5\sin t & \cos t & 0.7+0.2\sin t\\
0.8+0.2\sin t & 0.2+0.4\sin t & 0.7+0.2\sin t\\
0.9+0.1\sin t & 0.6+0.3\sin t & 0.4+0.4\cos t\\
\end{array}\right).\]
Obviously, $(H_{1})$ holds. Clearly, we have
\[M^{f}=M^{g}=0.05,\,\,\,L^{f}=L^{g}=0.05,\,\,\,
\sum\limits_{C_{kl}\in N_{1}(1,1)}\overline{B_{11}^{kl}}=\sum\limits_{C_{kl}\in N_{1}(1,1)}\overline{C_{11}^{kl}}=0.5,
\]
\[\sum\limits_{C_{kl}\in N_{1}(1,2)}\overline{B_{12}^{kl}}=\sum\limits_{C_{kl}\in N_{1}(1,2)}\overline{C_{12}^{kl}}=1,\,\,\,
\sum\limits_{C_{kl}\in N_{1}(1,3)}\overline{B_{13}^{kl}}=\sum\limits_{C_{kl}\in N_{1}(1,3)}\overline{C_{13}^{kl}}=0.6,
\]
\[\sum\limits_{C_{kl}\in N_{1}(2,1)}\overline{B_{21}^{kl}}=\sum\limits_{C_{kl}\in N_{1}(2,1)}\overline{C_{21}^{kl}}=0.8,\,\,\,
\sum\limits_{C_{kl}\in N_{1}(2,2)}\overline{B_{22}^{kl}}=\sum\limits_{C_{kl}\in N_{1}(2,2)}\overline{C_{22}^{kl}}=1.2,
\]
\[\sum\limits_{C_{kl}\in N_{1}(2,3)}\overline{B_{23}^{kl}}=\sum\limits_{C_{kl}\in N_{1}(2,3)}\overline{C_{23}^{kl}}=1,\,\,\,
\sum\limits_{C_{kl}\in N_{1}(3,1)}\overline{B_{31}^{kl}}=\sum\limits_{C_{kl}\in N_{1}(3,1)}\overline{C_{31}^{kl}}=0.5,
\]
\[\sum\limits_{C_{kl}\in N_{1}(3,2)}\overline{B_{32}^{kl}}=\sum\limits_{C_{kl}\in N_{1}(3,2)}\overline{C_{32}^{kl}}=0.8,\,\,\,
\sum\limits_{C_{kl}\in N_{1}(3,3)}\overline{B_{33}^{kl}}=\sum\limits_{C_{kl}\in N_{1}(3,3)}\overline{C_{33}^{kl}}=0.6
\]
and we can easily check that $\max\limits_{i,j}\frac{\overline{L_{ij}}}{\underline{a_{ij}}}=0.25$. Take $\rho=1$, then
\[\max\limits_{i,j}\bigg\{\frac{\sum\limits_{C_{kl}\in N_{r}(i,j)}\overline{B_{ij}^{kl}}M^{f}\rho^{2}
+\sum\limits_{C_{kl}\in N_{p}(i,j)}\overline{C_{ij}^{kl}}M^{g}\rho^{2}\overline{\delta_{kl}}+\overline{L_{ij}}}{\underline{a_{ij}}}
\bigg\}\approx 0.2669<\rho=1\]
and
\[\max\limits_{i,j}\bigg\{\frac{\sum\limits_{C_{kl}\in N_{r}(i,j)}\overline{B_{ij}^{kl}}(M^{f}+L^{f})\rho
+\sum\limits_{C_{kl}\in N_{p}(i,j)}\overline{C_{ij}^{kl}}(M^{g}+L^{g})\rho\overline{\delta_{kl}}}{\underline{a_{ij}}}\bigg\}\approx0.5337<1.\]
The combination of the above two inequalities means that $(H_{3})$ is satisfied for $\rho=1$.

Therefore, we have shown that assumptions $(H_{1})$-$(H_{3})$ are satisfied. By Theorem \ref{thm31}, system \eqref{g1} has exactly one almost automorphic solution in $\mathbb{E}=\{\varphi\in \mathbb{B}:\|\varphi\|_{\mathbb{B}}\leq \rho\}$. Moreover, by Theorem \ref{thm32}  this solution is globally exponentially stable.
\end{example}

\begin{example}
If $\mathbb{T}=\mathbb{Z}$, then $\mu(t)\equiv1$. Take
\[
\left(
\begin{array}{cccc}
a_{11}(t) & a_{12}(t) & a_{13}(t)\\
a_{21}(t) & a_{22}(t) & a_{23}(t)\\
a_{31}(t) & a_{32}(t) & a_{33}(t)\\
\end{array}\right)=\left(
\begin{array}{cccc}
0.3+0.1|\sin t| & 0.4+0.1|\cos \sqrt{2}t| & 0.2+0.1|\sin t|\\
0.2+0.2|\cos(\frac{1}{3} t)| & 0.5+0.1|\sin 2t| & 0.1+0.1|\cos 2t|\\
0.4+0.1|\cos2t| & 0.3+0.2|\sin 2t| & 0.2+0.1|\cos t|\\
\end{array}\right),\]
\begin{eqnarray*}
\left(
\begin{array}{cccc}
B_{11}(t) & B_{12}(t) & B_{13}(t)\\
B_{21}(t) & B_{22}(t) & B_{23}(t)\\
B_{31}(t) & B_{32}(t) & B_{33}(t)\\
\end{array}\right)&=&\left(
\begin{array}{cccc}
C_{11}(t) & C_{12}(t) & C_{13}(t)\\
C_{21}(t) & C_{22}(t) & C_{23}(t)\\
C_{31}(t) & C_{32}(t) & C_{33}(t)\\
\end{array}\right)\\&=&
\left(
\begin{array}{cccc}
0.02|\cos t| & 0.02|\cos 2t| & 0.03|\sin t|\\
0.01|\sin 2t| & 0.01|\cos t| & 0.02|\cos 2t|\\
0.03|\cos t| & 0.02|\cos 2t| & 0.01|\cos t|\\
\end{array}\right),
\end{eqnarray*}
\[
\left(
\begin{array}{cccc}
L_{11}(t) & L_{12}(t) & L_{13}(t)\\
L_{21}(t) & L_{22}(t) & L_{23}(t)\\
L_{31}(t) & L_{32}(t) & L_{33}(t)\\
\end{array}\right)=\left(
\begin{array}{cccc}
0.01\sin t & 0.01\cos t & 0.02\sin t\\
0.02\sin t & 0.04\sin t & 0.03\sin t\\
0.01\sin t & 0.03\sin t & 0.02\cos t\\
\end{array}\right),\]
\[
\left(
\begin{array}{cccc}
\delta_{11}(t) & \delta_{12}(t) & \delta_{13}(t)\\
\delta_{21}(t) & \delta_{22}(t) & \delta_{23}(t)\\
\delta_{31}(t) & \delta_{32}(t) & \delta_{33}(t)\\
\end{array}\right)=\left(
\begin{array}{cccc}
\sin t & 0.9\cos t & 0.2+0.5\sin t\\
0.2+0.8\sin t & 0.4+0.6\sin t & 0.1+0.9\sin t\\
\sin t & \sin t & 0.7+0.2\cos t\\
\end{array}\right).\]
Obviously, $(H_{1})$ holds. Clearly, we have
\[M^{f}=M^{g}=0.05,\,\,\,L^{f}=L^{g}=0.05,\,\,\,
\sum\limits_{C_{kl}\in N_{1}(1,1)}\overline{B_{11}^{kl}}=\sum\limits_{C_{kl}\in N_{1}(1,1)}\overline{C_{11}^{kl}}=0.06,
\]
\[\sum\limits_{C_{kl}\in N_{1}(1,2)}\overline{B_{12}^{kl}}=\sum\limits_{C_{kl}\in N_{1}(1,2)}\overline{C_{12}^{kl}}=0.11,\,\,\,
\sum\limits_{C_{kl}\in N_{1}(1,3)}\overline{B_{13}^{kl}}=\sum\limits_{C_{kl}\in N_{1}(1,3)}\overline{C_{13}^{kl}}=0.08,
\]
\[\sum\limits_{C_{kl}\in N_{1}(2,1)}\overline{B_{21}^{kl}}=\sum\limits_{C_{kl}\in N_{1}(2,1)}\overline{C_{21}^{kl}}=0.11,\,\,\,
\sum\limits_{C_{kl}\in N_{1}(2,2)}\overline{B_{22}^{kl}}=\sum\limits_{C_{kl}\in N_{1}(2,2)}\overline{C_{22}^{kl}}=0.17,
\]
\[\sum\limits_{C_{kl}\in N_{1}(2,3)}\overline{B_{23}^{kl}}=\sum\limits_{C_{kl}\in N_{1}(2,3)}\overline{C_{23}^{kl}}=0.11,\,\,\,
\sum\limits_{C_{kl}\in N_{1}(3,1)}\overline{B_{31}^{kl}}=\sum\limits_{C_{kl}\in N_{1}(3,1)}\overline{C_{31}^{kl}}=0.07,
\]
\[\sum\limits_{C_{kl}\in N_{1}(3,2)}\overline{B_{32}^{kl}}=\sum\limits_{C_{kl}\in N_{1}(3,2)}\overline{C_{32}^{kl}}=0.1,\,\,\,
\sum\limits_{C_{kl}\in N_{1}(3,3)}\overline{B_{33}^{kl}}=\sum\limits_{C_{kl}\in N_{1}(3,3)}\overline{C_{33}^{kl}}=0.06
\]
and we can easily check that $\max\limits_{i,j}\frac{\overline{L_{ij}}}{\underline{a_{ij}}}=0.15$. Take $\rho=1$, then
\[\max\limits_{i,j}\bigg\{\frac{\sum\limits_{C_{kl}\in N_{r}(i,j)}\overline{B_{ij}^{kl}}M^{f}\rho^{2}
+\sum\limits_{C_{kl}\in N_{p}(i,j)}\overline{C_{ij}^{kl}}M^{g}\rho^{2}\overline{\delta_{kl}}+\overline{L_{ij}}}{\underline{a_{ij}}}
\bigg\}\approx 0.1904<\rho=1\]
and
\[\max\limits_{i,j}\bigg\{\frac{\sum\limits_{C_{kl}\in N_{r}(i,j)}\overline{B_{ij}^{kl}}(M^{f}+L^{f})\rho
+\sum\limits_{C_{kl}\in N_{p}(i,j)}\overline{C_{ij}^{kl}}(M^{g}+L^{g})\rho\overline{\delta_{kl}}}{\underline{a_{ij}}}\bigg\}\approx0.3657<1.\]
The combination of the above two inequalities means that $(H_{3})$ is satisfied for $\rho=1$.

Therefore, we have shown that assumptions $(H_{1})$-$(H_{3})$ are satisfied. By Theorem \ref{thm31}, system \eqref{g1} has exactly one almost automorphic solution in $\mathbb{E}=\{\varphi\in \mathbb{B}:\|\varphi\|_{\mathbb{B}}\leq \rho\}$. Moreover, by Theorem \ref{thm32} this solution is globally exponentially stable.
 \end{example}

\section{Conclusion}
\indent

In this paper,  first a new concept of almost periodic time scales and    a new definition of almost
automorphic functions on almost periodic time scales are proposed and   some basic   properties of them are studied. Two open problems concerning the relationship between the algebraic operation property of elements of a time scale and
analytical property of the time scale are proposed.
Then,  based on these concepts and  results,   the existence of an almost automorphic solution for both the linear nonhomogeneous dynamic equation on time scales and its associated homogeneous equation is established.
Finally, as an application of the results,   the existence and global exponential stability of almost automorphic solutions to a class of shunting inhibitory cellular neural networks  with time-varying delays on time scales are obtained.

\end{document}